\documentclass[a4paper,10pt]{article}
\usepackage[utf8x]{inputenc}
\usepackage{amsmath,amsthm}
\usepackage{mathrsfs}
\usepackage{enumitem}
\usepackage{amsfonts}
\usepackage{amssymb}
\usepackage{hyperref}

\def\rmd{\mathrm{d}}
\def\define{:=}
\def\R{\mathbb{R}}
\providecommand{\abs}[1]{\left\lvert#1\right\rvert}
\providecommand{\norm}[1]{\left\lVert#1\right\rVert}
\providecommand{\set}[1]{\left\{#1\right\}}
\def\st{\ :\ }

\DeclareMathOperator{\print}{print}
\DeclareMathOperator{\diameter}{diam}
\DeclareMathOperator*{\esssup}{ess\;sup}
\DeclareMathOperator{\divergence}{div}
\DeclareMathOperator{\loc}{loc}
\DeclareMathOperator{\support}{supp}
\DeclareMathOperator{\closure}{cl}
\DeclareMathOperator{\curl}{curl}

\newtheorem{definition}{Definition}[section]
\newtheorem{lemma}[definition]{Lemma}
\newtheorem{theorem}[definition]{Theorem}
\newtheorem{proposition}[definition]{Proposition}
\newtheorem{corollary}[definition]{Corollary}
\newtheorem{rem}[definition]{Remark}

\title{On solutions of the transport equation in the presence of singularities.}

\author{Evelyne Miot\thanks{CNRS and Institut Fourier, Universit\'e Grenoble Alpes, France} \ and Nicholas Sharples\thanks{Middlesex University, UK}}

\begin{document}

\maketitle

\begin{abstract}
We consider the transport equation on $[0,T]\times \R^n$ in the situation where the vector field is $BV$ off a set $S\subset [0,T]\times \R^n$. We demonstrate that solutions exist and are unique provided that the set of singularities has a sufficiently small anisotropic fractal dimension and the normal component of the vector field is sufficiently integrable near the singularities. This result improves upon recent results of Ambrosio who requires the vector field to be of bounded variation everywhere.

In addition, we demonstrate that under these conditions almost every trajectory of the associated regular Lagrangian flow does not intersect the set $S$ of singularities.

Finally, we consider the particular case of an initial set of singularities that evolve in time so the singularities consists of curves in the phase space, which is typical in applications such as vortex dynamics. We demonstrate that solutions of the transport equation exist and are unique provided that the box-counting dimension of the singularities is bounded in terms of the H\"older exponent of the curves.
\end{abstract}

\section{Introduction}
In this note we are concerned with the existence and uniqueness of solutions to the transport equation
\begin{align}\label{transport equation}\tag{TE}
\left\lbrace
\begin{aligned}
\partial_{t}u+b\cdot \nabla u&=0 &\text{on}\;\left(0,T\right)\times\R^{n}\\
u\left(0,\cdot\right)&=u_{0}\left(\cdot\right)
\end{aligned}
\right.
\end{align}
for some $T>0$, when the non-autonomous vector field $b\colon\left[0,T\right]\times \R^{n}\rightarrow \R^{n}$ has limited regularity. Classically, the existence of unique smooth solutions to \eqref{transport equation} is assured if the vector field $b$ and the initial data $u_{0}$ are both smooth. However, in many applications such as Fluid Dynamics or Control Theory the smoothness or even continuity of vector fields cannot be guaranteed.

When considering less regular vector fields $b$, minimally requiring that both $b$ and its (spatial) divergence are locally integrable, we say that a bounded map $u$ is a weak solution of the transport equation if \eqref{transport equation} holds distributionally:
\begin{definition}
A map $u\in L^{\infty}_{\loc}\left(\left(0,T\right)\times \R^{n}\right)$ is a weak solution of \eqref{transport equation} with initial data $u_{0}\in L^{\infty}_{\loc}\left(\R^{n}\right)$ if
\begin{align}\label{weak transport equation}
\int_{\R^{n}} u_{0}\left(x\right)\phi\left(0,x\right) \rmd x + \int_{0}^{T}\int_{\R^{n}} u\cdot \left(\frac{\partial \phi}{\partial t} + b\cdot\nabla \phi + \divergence b\ \phi\right) \;\rmd x \;\rmd t = 0
\end{align}
for all test maps $\phi\in C^{\infty}_{c}\left(\left[0,T\right)\times \R^{n}\right)$.
\end{definition}

The transport equation \eqref{transport equation} corresponds with the ordinary differential equation
\begin{align}\label{ODE}\tag{ODE}
\frac{\rmd \xi}{\rmd t}&=b\left(t,\xi\right) && \xi\left(0\right)=x.
\end{align}
Classically, solutions of \eqref{transport equation} are obtained via the `method of characteristics' where the initial data $u_{0}$ is evolved along the flow solution of \eqref{ODE}.

In the less regular setting first considered by DiPerna \& Lions \cite{DiPernaLions89a} this correspondence reverses: solutions of \eqref{ODE} are obtained from solutions of \eqref{transport equation}. Here the appropriately weakened notion of a flow solution is that of a regular Lagrangian Flow (see  DiPerna \& Lions \cite{DiPernaLions89a}, Ambrosio \cite{Ambrosio04} and Crippa \& De Lellis \cite{CrippaDeLellis08}).

\begin{definition}
A map $X\colon \left[0,T\right]\times \R^{n}\rightarrow \R^{n}$ is a \emph{regular Lagrangian Flow solution} of \eqref{ODE} if
\begin{itemize}
\item for almost every $x\in\R^{n}$ the trajectory $t\mapsto X\left(t,x\right)$ is absolutely continuous and
\begin{equation}\label{trajectories}
X\left(t,x\right) = x + \int_{0}^{t} b\left(\tau,X\left(\tau,x\right)\right) \rmd \tau,
\end{equation}
and
\item there exists a constant $L>0$ such that for all Borel sets $B\subset \R^{n}$ the image measure
\begin{align}\label{compressibility constant}
\mu_{n}\left(X\left(t,\cdot\right)^{-1}\left(B\right)\right)&\leq L \mu_{n}\left(B\right) && \forall t\in\left[0,T\right]
\end{align}
where $\mu_{n}$ is the $n$-dimensional Lebesgue measure.
\end{itemize}
\end{definition}

As the trajectories are absolutely continuous the integral equality \eqref{trajectories} is equivalent to requiring $X\left(0,x\right)=x$ and $\frac{\rmd X}{\rmd t}=b\left(t,X\left(t,x\right)\right)$ for almost every $t\in\left[0,T\right]$. Further, the condition \eqref{compressibility constant} ensures that sets with positive measure do not evolve into sets with zero measure. From this fact it follows that $X$ is also a regular Lagrangian Flow solution of \eqref{ODE} for all vector fields $\tilde{b}$ that are equal to $b$ almost everywhere.

Solutions of \eqref{ODE} are obtained from solutions of \eqref{transport equation} via a `reverse method of characteristics': the existence, uniqueness and stability of regular Lagrangian flows follows from the existence, uniqueness and stability of weak solutions to \eqref{transport equation}. We refer to DiPerna \& Lions \cite{DiPernaLions89a} for the original proofs, Ambrosio \cite{Ambrosio04, Ambrosio08} for a more general approach, or Crippa \cite{Crippa08} or De Lellis \cite{DeLellis08a} for a more direct treatment in the case when $b$ is bounded.

Crippa \& De Lellis \cite{CrippaDeLellis08} provide an alternative approach to establishing existence and uniqueness of regular Lagrangian flows. The authors use the theory of maximal functions to obtain some new estimates on flows with Sobolev regularity, thereby obtaining uniqueness directly in the ODE framework. However, this approach requires slightly stronger regularity assumptions for the vector field $b$ than the one considered by Ambrosio \cite{Ambrosio04, Ambrosio08}.
We mention that a recent paper by Nguyen \cite{Nguyen} provides well-posedness of the flow associated to vector fields represented as singular integrals of $BV$ functions, by extending the estimates on the flow established by Crippa \& De Lellis \cite{CrippaDeLellis08}. In turn, this implies well-posedness of the corresponding continuity and transport equations.

\medskip

In this note we extend the theory of DiPerna \& Lions \cite{DiPernaLions89a} and Ambrosio \cite{Ambrosio04} to vector fields $b$ that are $BV$ off a set $S\subset \left[0,T\right]\times \R^{n}$. Let $d_{S}\colon \left[0,T\right]\times \R^{n}\rightarrow \R$ be the Euclidean distance function to $S$, that is
\[
d_{S}\left(t,x\right) \define \inf_{\left(s,y\right)\in S} \abs{\left(t,x\right)-\left(s,y\right)},
\]
for all $\varepsilon>0$ we write $\set{d_{S}>\varepsilon} \define \set{\left(t,x\right) \st d_{S}\left(t,x\right) > \varepsilon}$
and define sets with corresponding inequalities similarly.

Our main result is the following:
\begin{theorem}\label{theorem - all results}
Let $S\subset \left[0,T\right]\times\R^{n}$ be compact. If the vector field $b$ satisfies 
\begin{enumerate}[label=$\textup{\roman{*})}$,ref=$\roman{*})$]
\item $b\in L^{1}_{\loc}\left(\left(0,T\right)\times\R^{n}\right)$, \label{b integrable}
\item $\divergence b \in L^{1}\left(0,T;L^{\infty}\left(\R^{n}\right)\right)$, \label{div b bounded}
\item $\frac{b}{1+\abs{x}} \in L^{1}\left(\left(0,T\right)\times\R^{n}\right) + L^{1}\left(0,T;L^{\infty}\left(\R^{n}\right)\right)$, \label{b decay}
\item for all $\Omega \subset \subset S^{c}$, the restriction $b\vert_{\Omega} \in L^{1}\left(0,T;BV_{\loc}\left(\R^{n}\right)\right)$,\label{b renormalized on complement}
\item for some $1\leq p,q\leq \infty$
\begin{align}
b\cdot\nabla d_{S}&\in L^{p}\left(0,T;L^{q}_{\loc}\left(\R^{n}\right)\right)\label{b dot distance}\\
\text{and}\qquad d_{S}^{-1} &\in L^{p^{*}}\left(0,T;L^{q^{*}}_{\loc}\left(\R^{n}\right)\right) \label{distance reciprocal}
\end{align}
where 
$\frac{1}{p}+\frac{1}{p^{*}}=\frac{1}{q}+\frac{1}{q^{*}}=1$, \label{distance condition}
\end{enumerate}
then
\begin{itemize}
\item for all initial data $u_{0}\in L^{\infty}_{\loc}\left(\R^{n}\right)$ there exists a unique weak solution of \eqref{transport equation},
\item there exists a unique regular Lagrangian Flow solution $X$ of \eqref{ODE},
\item the regular Lagrangian Flow $X$ avoids the set $S$, that is
\begin{equation}
\mu_{n}\left(\set{x\in \R^{n} \st \left(t,X\left(t,x\right)\right)\in S\quad \text{for some}\; t\in \left[0,T\right]}\right) = 0.\label{avoidance}
\end{equation}
\end{itemize}
\end{theorem}

We prove this using the theory of renormalized solutions (after DiPerna \& Lions \cite{DiPernaLions89a} and Ambrosio \cite{Ambrosio04}) which we recall in Section \ref{section - existence and uniqueness} in their local formulation.

The avoidance result \eqref{avoidance} was first studied in the autonomous case by Aizenman \cite{Aizenman78a}, then Cipriano \& Cruzeiro \cite{CiprianoCruzeiro05} and Robinson \& Sharples \cite{RobinsonSharples13JDE} in the non-autonomous case, which we recall in Section \ref{section - avoidance}. In the present work we improve these result by accounting for the direction of $b$: the condition \eqref{b dot distance} only requires the component of $b$ \textit{normal} to $S$ to be integrable.

The condition \eqref{distance reciprocal} encodes some anisotropic fractal detail of the set $S$, first studied in Robinson \& Sharples \cite{RobinsonSharples13JDE}. In Section \ref{section - fractal geometry} we recall the basic properties of the `codimension print', and its relationship to the more familiar box-counting dimensions. In Section \ref{section - codimension of trajectories} we study the codimension prints for sets $S$ consisting of singularities that evolve with time, obtaining the following result in terms of the box-counting dimensions of the temporal sections $S\left(t\right) \define \set{x\in \R^{n} \st \left(t,x\right) \in S}$ of the singular set $S$.

\begin{proposition}\label{prop:traj}
Suppose $S_{0}$ is compact, let $Z\colon\left[0,T\right]\times S_{0} \rightarrow \R^{n}$ and suppose there exists $\alpha$ in the range $0<\alpha \leq 1$ and $K>0$ such that
\begin{align*}
\abs{Z\left(t_{1},x\right)-Z\left(t_{2},x\right)} &\leq K \abs{t_{1}-t_{2}}^{\alpha} &\forall t_{1},t_{2}\in\left[0,T\right]\quad \forall x\in S_{0}
\end{align*}
(i.e. $Z$ is $\alpha$-H\"older continuous in $t$, uniformly in $x$).

Let $S=\set{\left(t,Z\left(t,x\right)\right) \st t\in\left[0,T\right],\  x\in S_{0}}$. If $b$ satisfies \ref{b integrable}, \ref{div b bounded}, \ref{b decay}, \ref{b renormalized on complement}, and
\[
b\cdot \nabla d_{S} \in L^{1}\left(0,T;L^{q}_{\loc}\left(\R^{n}\right)\right)
\]
for some $q$ with
\[
\frac{1}{q} + \frac{1}{\alpha\left(n-\sup_{t\in\left[0,T\right]} \dim_{\rm B} S\left(t\right)\right)}<1
\]

then the conclusions of Theorem \ref{theorem - all results} hold. Here, $\dim_{\rm B} S\left(t\right)$ denotes the box-counting dimension of the set $S(t)$, which is defined in Definition \ref{def:box} below.
\end{proposition}

\subsection{Existence and Uniqueness}\label{section - existence and uniqueness}
The existence of weak solutions of \eqref{transport equation} requires only the additional assumption that the (spatial) divergence $\divergence b \in L^{1}\left(0,T;L^{\infty}\left(\R^{n}\right)\right)$ and follows from a standard compactness argument (see, for example, DiPerna \& Lions \cite{DiPernaLions89a} Proposition II.1).
However without further regularity assumptions the uniqueness of weak solutions is not assured. Indeed, Depauw \cite{Depauw03} constructs a divergenceless, bounded vector field $b$ that, in addition to the trivial zero solution, admits a weak solution $u\neq 0$ with initial data $u_{0}= 0$.

In their seminal paper DiPerna \& Lions \cite{DiPernaLions89a} proved that weak solutions of \eqref{transport equation} are unique under the additional assumptions that $b$ has some Sobolev regularity (i.e. $b$ is integrable and has integrable weak derivatives), and is either bounded or decays sufficiently quickly at infinity. De Lellis \cite{DeLellis08} comments that DiPerna \& Lions' strategy can be decomposed into an `easy' part and a `hard' part. The easy part establishes uniqueness provided that every weak solution also satisfies (in a distributional sense) the `renormalized' equation
\begin{align}\label{renormalized transport equation}
\partial_{t}\beta\left(u\right)+b\cdot \nabla \beta\left(u\right)&=0 &\text{on}\;\left(0,T\right)\times\R^{n}
\end{align}
for all maps $\beta$ in an appropriate class, which trivially holds if $u$ is a smooth solution. Formally, by setting $\beta\left(z\right)=z^{2}$, integrating \eqref{renormalized transport equation}, and applying the divergence theorem we obtain
\begin{align}\label{formal gronwall}
\frac{\rmd}{\rmd t} \int_{\R^{n}} u^{2}\left(t,x\right) \rmd x \leq \norm{\divergence b\left(t\right)}_{L^{\infty}\left(\R^{n}\right)} \int_{\R^{n}}u^{2}\left(t,x\right)\rmd x,
\end{align}
from which we conclude with Gronwall's inequality that $u\equiv 0$ for initial data $u_{0}=0$. The uniqueness of solutions then follows from the linearity of \eqref{renormalized transport equation}. DiPerna \& Lions' results make this argument rigorous for weak solutions by employing a distributional form of Gronwall's inequality (Theorem II.2 of DiPerna \& Lions \cite{DiPernaLions89a}).

The hard part of the DiPerna-Lions theory is to demonstrate that all weak solutions satisfy the renormalized equation \eqref{renormalized transport equation} provided that $b$ has Sobolev regularity $b\in L^{1}(0,T;W_{\loc}^{1,1}(\R^{n}))$. Essentially, this follows from noting that for each weak solution $u$, its mollification $u_{\varepsilon}=u\ast \rho_{\varepsilon}$ satisfies $\partial_{t}u_{\varepsilon}+b\cdot \nabla u_{\varepsilon} = r_{\varepsilon}$ where 
$\rho_{\varepsilon}$ is the standard (spatial) mollifier and $r_{\varepsilon}=b\cdot \nabla u_{\varepsilon} - \left(b\cdot \nabla u\right)\ast \rho_{\varepsilon}$ is the commutator of the second term of \eqref{transport equation} with respect to the mollifier. As the mollified solution is smooth, it follows that $\partial_{t} \beta\left(u_{\varepsilon}\right)+b\cdot \nabla \beta\left(u_{\varepsilon}\right)=\beta^{\prime}\left(u_{\varepsilon}\right) r_{\varepsilon}$ for arbitrary $\beta\in C^{1}\left(\R\right)$. Passing to the limit as $\varepsilon\rightarrow 0$ we conclude that $u$ is a renormalized solution provided that the commutator $r_{\varepsilon}\rightarrow 0$ in $L^{1}_{\loc}\left(\R^{n}\right)$. The technical portion of DiPerna \& Lions' uniqueness result is precisely that the commutator $r_{\varepsilon}$ converges for Sobolev vector fields $b$ (Lemma II.1 of DiPerna \& Lions \cite{DiPernaLions89a}).

The DiPerna-Lions theory has been adapted to a wider class of vector fields including those that are `piecewise Sobolev' (Lions \cite{Lions98}), of Vlasov type (Bouchut \cite{Bouchut01}), or have `conormal BV' regularity (Colombini \& Lerner \cite{colombinilerner05}). In these cases, without Sobolev regularity of the vector field, the convergence of the commutator $r_{\varepsilon}$ is highly sensitive to the choice of mollifier $\rho_{\varepsilon}$ (see Capuzzo Dolcetta \& Perthame\cite{CapuzzoDolcettaPerthame96}). These results rely on an anisotropic smoothing argument, in which the mollifier $\rho_{\varepsilon}$ is locally chosen to account for the particular structure of the vector field (see \cite{Ambrosio04}).

In a significant breakthrough, Ambrosio \cite{Ambrosio04} extended DiPerna \& Lions' theory to the large class of vector fields of bounded variation $b\in L^{1}\left(0,T;BV_{\loc}\left(\R^{n}\right)\right)$ (i.e. the spatial distributional derivative of $b$ is a measure with finite total variation), which includes the classes considered in \cite{Lions98}, \cite{Bouchut01}, and \cite{colombinilerner05}. Ambrosio's highly technical analysis uses Alberti's rank one theorem, a deep measure theoretic result, to show that at small scales any BV vector field behaves like one of the `conormal' BV fields considered in \cite{colombinilerner05}.

In the present work we consider a class of vector fields that are locally in $L^{1}\left(0,T;BV_{\loc}\left(\R^{n}\right)\right)$ \emph{except} on a set $S\subset \left[0,T\right]\times \R^{n}$ (we make this precise later). Ultimately we will require that the `singular set' $S$ will have a small anisotropic fractal dimension in the sense of Robinson \& Sharples \cite{RobinsonSharples13JDE}, which is related to the familiar box-counting dimensions.

Vector fields of this type appear quite naturally in point vortex models of fluid dynamics. For example, Crippa et al. \cite{CrippaLopesFilhoMiotNussenzveigLopes16} consider the vortex-wave system: in this $2$-dimensional setting the Biot-Savart law is used to recover the velocity field $b$ of a fluid from its vorticity $\omega=\curl b$, which includes an initial dirac mass at $z_{0}\in \R^{2}$ that evolves along a Lipschitz trajectory $t\mapsto z\left(t\right)$. The resulting velocity field $b\left(t,x\right)=v(t,x)+\left(x-z(t)\right)^{\perp}/\abs{x-z(t)}^{2}$, with $v$ bounded and enjoying spatial Sobolev regularity, does not have bounded variation (nor finite $L^{2}$ norm) in any neighbourhood of the trajectory of the dirac mass $S=\set{\left(t,z\left(t\right)\right) \st t\in\left[0,T\right]}$, therefore falling outside the scope of Ambrosio's uniqueness result. Nevertheless, exploiting the explicit form of the singular part of $b$, it is proved in \cite{CrippaLopesFilhoMiotNussenzveigLopes16} that there exists a unique regular Lagrangian flow and that generically, its trajectories do not intersect the trajectory of the point vortex. Hence Theorem \ref{theorem - all results} can be seen as an extension of that result for more general singular sets.
In Subsection \ref{subsec:vortex-wave} we will explain more in detail the link between the present work and \cite{CrippaLopesFilhoMiotNussenzveigLopes16}. In particular, the avoidance result of Theorem \ref{theorem - all results} may be simply adapted to retrieve the result of Crippa et al. \cite{CrippaLopesFilhoMiotNussenzveigLopes16} in a straightforward way.
Similar questions also occur in the setting of Vlasov-Poisson equations with densities including point charges that generate point singularities in the electric field. We refer to the articles by Caprino et al. \cite{caprino} for the two-dimensional case and by Crippa, Ligabue, Saffirio \cite{Ligabue} for the three-dimensional case.
Finally,  in connection with the study of singular velocities generated by singular measure-valued vorticities in fluid dynamics, we mention the recent work by Fefferman, Pooley and Rodrigo \cite{Fefferman} on the construction of velocity fields for active scalar systems with measure-valued solutions whose support does not satisfy conservation of the Hausdorff dimension.
\subsubsection{Local renormalisation}
We now recall the local formulation of the renormalization theory (see DiPerna \& Lions \cite{DiPernaLions89a}, Ambrosio \cite{Ambrosio04} and De Lellis \cite{DeLellis08}).

\begin{definition}\label{definition - renormalized solution}
A weak solution $u\in L^{\infty}_{\loc}\left(\left(0,T\right)\times\R^{n}\right)$ of \eqref{transport equation} with initial data $u_{0}\in L^{\infty}_{\loc}\left(\R^{n}\right)$ is said to be \emph{renormalized} if for all $\beta\in C^{1}\left(\R\right)$ the map $\beta\left(u\right)$ is a weak solution of \eqref{transport equation} with initial data $\beta\left(u_{0}\right)$.
\end{definition}
We say that $b$ has the \emph{renormalization property} if every weak solution of \eqref{transport equation} is a renormalized solution.

The formal uniqueness argument in the previous section holds for renormalized solutions:

\begin{theorem}[\cite{DiPernaLions89a} Theorem II.2]\label{theorem - DiPLi}
If the vector field $b$ satisfies \ref{b integrable}, \ref{div b bounded} and \ref{b decay} and $b$ has the renormalization property then for all $u_{0}\in L^{\infty}_{\loc}\left(\R^{n}\right)$ there exists a unique solution to \eqref{transport equation}.
\end{theorem}

To prove our main result we proceed locally: we will show that $b$ has the renormalization property `away from $S$' where it has bounded variation, and also in a neighbourhood of $S$, as the set is sufficiently small and $b$ sufficiently integrable.

To define renormalization on an open set $\Omega\subset \left[0,T\right]\times \R^{n}$ we must suppress the requirement for the initial condition $\beta\left(u\left(0,\cdot\right)\right)=\beta\left(u_{0}\right)$ to be satisfied (as indeed $t=0$ may not intersect $\Omega$). The transformed initial condition, which is necessary for the Gronwall argument of Theorem \ref{theorem - DiPLi}, can then be recovered by extending $b$ to negative time.

\begin{definition}\label{definition - locally renormalized solution}
A weak solution $u\in L^{\infty}_{\loc}\left(\left(0,T\right)\times\R^{n}\right)$ of \eqref{transport equation} is said to be \emph{locally renormalized} on an open subset $\Omega\subset \left(0,T\right)\times\R^{n}$ if for all $\beta\in C^{1}\left(\R\right)$
\begin{equation}\label{renormalization property}
\iint_{\Omega} \beta\left(u\right) \left(\frac{\partial\phi}{\partial t} + b\cdot \nabla \phi + \divergence b\;\phi\; \right)\rmd x\;\rmd t=0
\end{equation}
for all test maps $\phi\in C_{c}^{\infty}\left(\Omega\right)$.
\end{definition}

Again, we say that $b$ has the \emph{local renormalization property} on $\Omega$ if every weak solution of \eqref{transport equation} is locally renormalized on $\Omega$.

\begin{rem}\label{remark - Lipschitz test functions}
If a weak solution $u$ is locally renormalized on $\Omega$ then \eqref{renormalization property} in fact holds for all compactly support Lipschitz maps $\phi\colon \Omega\rightarrow \R$. This follows by approximating $\phi$ by $\phi_{\varepsilon} = \phi \ast \rho_{\varepsilon}$ where $\rho_{\varepsilon}$ is the standard mollifier (see, for example, \S 4.2 of Evans \& Gariepy \cite{EvansGariepy92}). The derivative of the mollified $\phi$
\[
\nabla \phi_{\varepsilon} = \nabla \left(\phi\ast \rho_{\varepsilon}\right) = \left(\nabla \phi\right)_{\varepsilon} \rightarrow \nabla \phi
\]
pointwise as $\varepsilon\rightarrow 0$ wherever the derivative $\nabla \phi$ exists, which is almost everywhere by Rademacher's Theorem. Further, as $\phi$ is Lipschitz the derivative $\nabla\phi$ is bounded so $\norm{\left(\nabla \phi\right)_{\varepsilon}}_{L^{\infty}\left(\Omega_{\varepsilon}\right)} \leq \norm{\nabla \phi}_{L^{\infty}\left(\Omega\right)}$. Consequently,
\[
\abs{\beta\left(u\right) b\cdot\nabla \phi_{\varepsilon}} \leq \abs{\beta\left(u\right) b\cdot\nabla \phi} \in L^{1}\left(\Omega\right)
\]
so by the Lebesgue Dominated Convergence Theorem $\beta\left(u\right) b\cdot\nabla \phi_{\varepsilon} \rightarrow \beta\left(u\right) b\cdot\nabla \phi$ in $L^{1}\left(\Omega\right)$.
A similar treatment for the remaining terms in \eqref{renormalization property} gives
\begin{multline*}
0 = \iint_{\Omega} \beta\left(u\right) \left(\frac{\partial\phi_{\varepsilon}}{\partial t} + b\cdot \nabla \phi_{\varepsilon} + \divergence b\;\phi_{\varepsilon}\; \right)\rmd x\;\rmd t \\ \rightarrow \iint_{\Omega} \beta\left(u\right) \left(\frac{\partial\phi}{\partial t} + b\cdot \nabla \phi + \divergence b\;\phi\; \right)\rmd x\;\rmd t
\end{multline*}
where the first equality holds as $\phi_{\varepsilon} \in C^{\infty}_{c}\left(\Omega\right)$ and $u$ is a locally renormalised solution of \eqref{transport equation} on $\Omega$. Consequently \eqref{renormalization property} holds for compactly supported Lipschitz $\phi$.
\end{rem}

The non-uniqueness example of Depauw \cite{Depauw03} illustrates that local renormalization on $\left(0,T\right)\times \R^{n}$ is not sufficient for renormalization: the author constructs a vector field $b$ and a solution $u$ with initial data $u_{0}$ such that $\beta\left(u\right)$ is a weak solution (hence $u$ is locally renormalized on $\left(0,T\right)\times \R^{n}$) with initial data distinct from $\beta\left(u_{0}\right)$ (hence $u$ is not renormalized).

We can recover renormalization from local renormalization by using the following `trick' made explicit in De Lellis.

\begin{lemma}[De Lellis \cite{DeLellis08a} \S 2.3]
Let $b\in L^{1}_{\loc}\left(\left(0,T\right)\times\R^{n}\right)$ and extend $b$ to negative time with $b\left(t,\cdot\right)\equiv 0$ for $t<0$.

If $b$ is locally renormalized on $\left(-\infty,T\right)\times \R^{n}$, then
$b$ is renormalized.
\end{lemma}

Finally, we give a local statement of Ambrosio's renormalisation result for vector fields of bounded variation (Ambrosio's proof proceeds locally, although the statement he gives is global).

\begin{theorem}[Ambrosio \cite{Ambrosio04} Theorem 3.5]
Let $b\in L^{1}_{\loc}\left(\left(0,T\right)\times\R^{n}\right)$ and extend $b$ to negative time with $b\left(t,\cdot\right)\equiv 0$ for $t<0$.
Let $\Omega \subset \left(-\infty,T\right)\times \R^{n}$ be an open set. If
\begin{align*}
b\vert_{\Omega} \in L^{1}_{\loc}\left(-\infty,T;BV_{\loc}\left(\R^{n}\right)\right)\\
\text{and}\qquad \divergence b \in L^{1}_{\loc}\left(\left(-\infty,T\right)\times \R^{n}\right)
\end{align*}
then $b$ has the local renormalization property on $\Omega$.
\end{theorem}

\subsection{Fractal geometry}\label{section - fractal geometry}
With many evolutionary differential equations, including the transport equation \eqref{transport equation}, it is natural to distinguish between spatial and temporal regularity. This manifests in the Bochner spaces on which the vector field $b$ is defined.

We will make a similar distinction between the spatial and temporal detail of the set of non-BV singularities $S\subset \left[0,T\right] \times \R^{n}$ using some tools of fractal geometry. In particular we will use the codimension print of Robinson \& Sharples \cite{RobinsonSharples13JDE}, which we recall below. This will be particularly useful for singular sets composed of trajectories (such as moving point vortices), which is the content of Theorem \ref{theorem - singular trajectories}.

The familiar Hausdorff and box-counting dimensions (recalled below) fail to encode any anisotropic (i.e. directionally dependent) detail of a set: for example if $C$ is the Cantor `middle half' set, which has Hausdorff and box-counting dimensions equal to $\frac{1}{2}$, then the product set $C\times C\subset \R^{2}$ has Hausdorff and box-counting dimensions equal to $1$ (see Example 7.6 in \cite{BkFalconer03}). Consequently, these standard notions of fractal dimension are unable to distinguish between the product set $C\times C$ and a line segment.

The anisotropic fractal detail of subsets was first considered by Rogers \cite{Rogers88} who adapts the Hausdorff dimension by considering a family of Hausdorff measures $\mathcal{H}^{\alpha}$ on $\R^{n}$ parameterised by $\alpha\in \R_{+}^{n}$, rather than the usual $1$-parameter family of Hausdorff measures. Rogers then encodes the detail of a subset $A\subset \R^{n}$ in a `Hausdorff dimension print', defined as the set of $\alpha\in \R_{+}^{n}$ such that $\mathcal{H}^{\alpha}\left(A\right)>0$.

The codimension print similarly encodes the anisotropic detail of $S$ by considering the integrability of $d_{S}^{-1}$, the reciprocal of the distance function.
\begin{definition}
For a subset $S\subset \left[0,T\right]\times\R^{n}$ the codimension print of $S$ is the subset
\[
\print \left(S\right)\define \set{\left(\alpha,\beta\right)\in (0,\infty]^{2}\  \st  d_{S}^{-1} \in L^{\beta}\left(0,T;L^{\alpha}_{\loc}\left(\R^{n}\right)\right) }.
\]
\end{definition}

\begin{rem}
In the previous definition, $\alpha$ and $\beta$ may belong to the interval $(0,\infty]$. Our main result entails that the exponents belong to $[1,\infty]$, see Theorem \ref{theorem - existence and uniqueness}, so we will throughout use values of $(\alpha,\beta)$ in the range $[1,\infty]$.
\end{rem}
Immediately we see that $\print(S)$ is empty if the $n+1$ dimensional Lebesgue measure $\mu_{n+1}\left(S\right)>0$. Other basic properties of the codimension print are as follows:

\begin{lemma}[Robinson \& Sharples \cite{RobinsonSharples13JDE} Lemma 3.1]\label{lemma - basic codimension properties}
The codimension print reverses inclusions and is invariant under closure of sets, that is for bounded sets
$S,S_{1},S_{2}\subset \left[0,T\right]\times\R^{n}$ 
\begin{align*}
S_{1}\subset S_{2} &\Rightarrow \print\left(S_{2}\right)\subset\print\left(S_{1}\right)& \\
\text{and}\qquad \print\left(\closure\left(S\right)\right) &= \print\left(S\right).
\end{align*}
\end{lemma}
The reversal of inclusions property justifies the use of the term `codimension' as this property is shared by the more familiar codimensions $n-\dim A$ for $A\subset\R^{n}$, which appears in the `Minkowski Sausage' formulation of the box-counting dimension (see below).

Computing the codimension print of even elementary sets can be quite involved (see Robinson \& Sharples \cite{RobinsonSharples13JDE} Example 3.5 for the codimension print of a singleton set). However, a portion of the codimension print can be recovered from the more elementary box-counting dimensions of the set, together with its projections.

\begin{definition}\label{def:box}
The upper and lower box-counting dimensions of a bounded set $A\subset \R^{n}$ are given by
\begin{align*}
\dim_{\rm B}A &\define \limsup_{\varepsilon\rightarrow 0} \frac{\log N\!\left(A,\varepsilon\right)}{-\log\varepsilon}\\
\dim_{\rm LB}A &\define \liminf_{\varepsilon\rightarrow 0} \frac{\log N\!\left(A,\varepsilon\right)}{-\log\varepsilon}
\end{align*}
respectively, where $N\left(A,\varepsilon\right)$ is the smallest number of sets with diameter at most $\varepsilon$ that form a cover of $A$, or one of many similar quantities which give an equivalent definition (discussed in Falconer \cite{BkFalconer03} \S 3.1 `Equivalent Definitions').
\end{definition}

Another useful formulation is given in terms of the Lebesgue measure of the $\varepsilon$-neighbourhoods of $A$
\[
\set{d_{A} < \varepsilon} \define \set{x\in\R^{n} \st d_{A}(x)< \varepsilon}.
\]

\begin{lemma}[`Minkowski Sausage' formulation]\label{lemma - Minkowski Sausage}
The upper and lower box-counting dimensions of a bounded set $A\subset \R^{n}$ are given by
\begin{align*}
\dim_{\rm B}A &= n-\liminf_{\varepsilon\rightarrow 0} \frac{\log \mu_{n}\left(\set{d_{A}<\varepsilon}\right)}{\log\varepsilon}\\
\dim_{\rm LB}A &= n-\limsup_{\varepsilon\rightarrow 0} \frac{\log \mu_{n}\left(\set{d_{A}<\varepsilon}\right)}{\log\varepsilon}
\end{align*}
\end{lemma}
\begin{proof}
See Falconer \cite{BkFalconer03} \S 3.1 `Equivalent Definitions'.
\end{proof}

The box-counting dimension of $S$ gives some of the `isotropic' component of $\print(S)$:
\begin{theorem}\label{theorem - isotropic print}
For a bounded subset $S\subset \left[0,T\right]\times \R^{n}$
\begin{align*}
\alpha &< n+1 -\dim_{\rm B}S & &\Rightarrow & \left(\alpha,\alpha\right) &\in \print(S) \\
\alpha &> n+1 -\dim_{\rm LB}S & &\Rightarrow & \left(\alpha,\alpha\right) &\notin \print(S) \\
\end{align*}
\end{theorem}
\begin{proof}
Follows from Remark 1 of \cite{Aizenman78a}.
\end{proof}

For product sets we can get some of the `anisotropic' component of $\print(S)$ from the box-counting dimensions of the component sets:

\begin{theorem}[Robinson \& Sharples \cite{RobinsonSharples13JDE} Theorem 3.4]\label{theorem - codimension product print}
For bounded subsets $\mathcal{T}\subset \left[0,T\right]$ and $A\subset \R^{n}$ the point $\left(\alpha,\beta\right)\in \print \left(\mathcal{T}\times A\right)$ if one of the following conditions holds:
\begin{itemize}
\item $\alpha < n-\dim_{\rm B}A$,
\item $\beta<1-\dim_{\rm B}\mathcal{T}$,
\item $\alpha\beta<\alpha\left(1-\dim_{\rm B}\mathcal{T}\right)+\beta\left(n-\dim_{\rm B}A\right)$.
\end{itemize}
Further, the point $\left(\alpha,\beta\right)\notin \print \left(\mathcal{T}\times A\right)$ if
\begin{itemize}
 \item $\alpha\beta>\alpha\left(1-\dim_{\rm LB}\mathcal{T}\right)+\beta\left(n-\dim_{\rm LB}A\right)$.
\end{itemize}
\end{theorem}

We remark that this theorem does not completely supersede that of Theorem \ref{theorem - isotropic print}: for the line $\alpha=\beta$ it follows from Theorem 1.7 that
\begin{align*}
\alpha &< 1+n-\left(\dim_{\rm B}\mathcal{T} + \dim_{\rm B}A \right)&  &\Rightarrow  &\left(\alpha,\alpha\right)&\in \print(S) \\
\alpha &> 1+n-\left(\dim_{\rm LB}\mathcal{T} + \dim_{\rm LB}A \right)&  &\Rightarrow  &\left(\alpha,\alpha\right)&\notin \print(S)
\end{align*}
which is weaker than Theorem \ref{theorem - isotropic print} as the box-counting product inequalities
\[
\dim_{\rm LB} \mathcal{T} + \dim_{\rm LB} A \leq \dim_{\rm LB} \left(\mathcal{T}\times A\right) \leq \dim_{\rm B} \left(\mathcal{T}\times A\right) \leq \dim_{\rm B} \mathcal{T} + \dim_{\rm B} A
\]
can be strict (see, Robinson \& Sharples \cite{RobinsonSharples13RAEX} and Example 3.6 of \cite{RobinsonSharples13JDE}).

Finally, we interpret this product set result in terms of the projections of the set $S$.
\begin{corollary}
For a bounded subset $S\subset \left[0,T\right]\times \R^{n}$ the point $\left(\alpha,\beta\right) \in \print\left(S\right)$ if one of the following holds:
\begin{itemize}
\item $\alpha < n-\dim_{\rm B}P_{x}(S)$,
\item $\beta<1-\dim_{\rm B}P_{t}(S)$,
\item $\alpha\beta<\alpha\left(1-\dim_{\rm B}P_{t}(S)\right)+\beta\left(n-\dim_{\rm B}P_{x}(S)\right)$.
\end{itemize}
where $P_{t}(S)$ and $P_{x}(S)$ are the temporal and spatial projections of $S$ respectively.
\end{corollary}
\begin{proof}
Follows from the inclusion $S\subset P_{t}(S) \times P_{x}(S)$, Lemma \ref{lemma - basic codimension properties} and Theorem \ref{theorem - codimension product print}.
\end{proof}

\subsection{Avoidance}\label{section - avoidance}
In the classical framework for ODEs, Aizenman \cite{Aizenman78a} considered vector fields $b$ that are smooth (or Lipschitz) on the complement of some singular set.

In general there is no flow solution of \eqref{ODE} in this setting as typically some trajectories will intersect $S$. However if almost every trajectory does not intersect (the closure of) $S$ then, as the trajectories are unique and defined for all time, this aggregate of trajectories gives a unique flow solution defined almost everywhere.

To formalise this argument, Aizenman considered an aggregate of `local' trajectories together with their existence times, and provided conditions for almost every `local' trajectory to avoid the set $S$.

As we will obtain existence of solutions using the renormalisation methods, we can instead for convenience define avoidance of sets in terms of an existing regular Lagrangian flow:

\begin{definition}
A regular Lagrangian flow $X\colon \left[0,T\right]\times \R^{n}\rightarrow \R^{n}$ avoids a closed set $S\subset \left[0,T\right]\times \R^{n}$ if
\[
\mu_{n}\left(\set{x\in \R^{n}\st  \left(t,X\left(t,x\right)\right)\in S\ \text{for some}\ t\in\left[0,T\right]}\right) = 0.
\]
\end{definition}

In these terms, Aizenman proved the following `autonomous' avoidance result:
\begin{theorem}[Aizenman \cite{Aizenman78a}]
Let $b\in L^{q}_{\loc}\left(\R^{n}\right)$. A regular Lagrangian flow solution $X\colon \left[0,T\right]\times \R^{n} \rightarrow \R^{n}$ of \eqref{ODE} avoids a set $\left[0,T\right]\times A$ if
\[
\frac{1}{q} + \frac{1}{n-\dim_{\rm B} A} < 1.
\]
\end{theorem}

In Robinson \& Sharples \cite{RobinsonSharples13JDE} this result was adapted to the non-autonomous setting:
\begin{theorem}[Robinson \& Sharples \cite{RobinsonSharples13JDE}]\label{theorem - avoidance}
A regular Lagrangian flow solution $X\colon \left[0,T\right]\times \R^{n} \rightarrow \R^{n}$ of \eqref{ODE} avoids a set $S\subset \left[0,T\right]\times \R^{n}$ if
\begin{equation}\label{avoidance criterion}
b \in L^{p}\left(0,T;L^{q}_{\loc}\left(\R^{n}\right)\right)
\end{equation}
and
\[
d_{S}^{-1} \in L^{p^{*}}\left(0,T;L^{q^{*}}\left(\R^{n}\right)\right)
\]
(i.e. $\left(q^{*},p^{*}\right) \in \print(S)$ )
where $\frac{1}{p}+\frac{1}{p^{*}} = \frac{1}{q}+\frac{1}{q^{*}} = 1$.
\end{theorem}

In the present work we improve this result to account for the direction of the vector field near the set $S$. This is appropriate for the analysis of point-vortices as the Biot-Savart law generates a vector field perpendicular to the singular set.

In Section \ref{section - main results} we prove the following.
\begin{theorem}
A regular Lagrangian flow solution $X\colon \left[0,T\right]\times \R^{n} \rightarrow \R^{n}$ of \eqref{ODE} avoids a set $S\subset \left[0,T\right]\times \R^{n}$ if
\[
b \cdot \nabla d_{S} \in L^{p}\left(0,T;L^{q}_{\loc}\left(\R^{n}\right)\right),
\]
and
\[
d_{S}^{-1} \in L^{p^{*}}\left(0,T;L^{q^{*}}\left(\R^{n}\right)\right)
\]
(i.e. $\left(q^{*},p^{*}\right) \in \print(S)$ )
where $\frac{1}{p}+\frac{1}{p^{*}} = \frac{1}{q}+\frac{1}{q^{*}} = 1$.
\end{theorem}

This result relies on the validity of the chain rule 
\begin{equation}\label{chain rule}
\frac{\rmd }{\rmd t} d_{S}\left(t,X\left(t,x\right)\right) = \frac{\partial}{\partial t} d_{S}\left(t,X\left(t,x\right)\right) + \frac{\partial}{\partial t} X\left(t,x\right) \cdot \nabla d_{S}\left(t,X\left(t,x\right)\right)
\end{equation}
which is not immediate as $d_{S}$ is only Lipschitz continuous and $X$ is only absolutely continuous in $t$ for almost every $x\in \R^{n}$.

Avoidance results can yield interesting qualitative properties of a regular Lagrangian flow: Robinson \& Sadowski \cite{RobinsonSadowski09a} demonstrate the almost everywhere uniqueness of particle trajectories for suitable weak solutions of the Navier-Stokes equations using avoidance methods (See also \cite{RobinsonSadowski09} and \cite{RobinsonSadowskiSharples12}).

\section{Proofs of the main results}\label{section - main results}
\subsection{Existence and Uniqueness of Solutions}
We start by proving the first part of Theorem \ref{theorem - all results}, which may be formulated as follows:
\begin{theorem}\label{theorem - existence and uniqueness}
Let $S\subset \left[0,T\right]\times\R^{n}$ be compact. If the vector field $b$ satisfies 
\begin{enumerate}[label=$\textup{\roman{*})}$,ref=$\roman{*})$]
\item $b\in L^{1}_{\loc}\left(\left(0,T\right)\times\R^{n}\right)$,
\item $\divergence b \in L^{1}\left(0,T;L^{\infty}\left(\R^{n}\right)\right)$, 
\item $\frac{b}{1+\abs{x}} \in L^{1}\left(\left(0,T\right)\times\R^{n}\right) + L^{1}\left(0,T;L^{\infty}\left(\R^{n}\right)\right)$,
\item for all $\Omega \subset \subset S^{c}$, the vector field $b$ extended by $b\left(t,\cdot\right) \equiv 0$ for $t<0$ has the renormalization property on $\Omega$.

\item for some $1\leq p,q\leq \infty$
\begin{align}
b\cdot\nabla d_{S}&\in L^{p}\left(0,T;L^{q}_{\loc}\left(\R^{n}\right)\right)
&\text{and}\qquad  d_{S}^{-1} &\in L^{p^{*}}\left(0,T;L^{q^{*}}_{\loc}\left(\R^{n}\right)\right) 
\end{align}
(i.e. $\left(q^{*},p^{*}\right) \in \print(S)$)
where 
$\frac{1}{p}+\frac{1}{p^{*}}=\frac{1}{q}+\frac{1}{q^{*}}=1$,
\end{enumerate}
then for all initial data $u_{0}\in L^{\infty}_{\loc}\left(\R^{n}\right)$ there exists a unique weak solution of \eqref{transport equation}.
\end{theorem}

\begin{corollary}
With the same hypotheses, there exists a unique regular Lagrangian flow solution of \eqref{ODE}.
\end{corollary}

\begin{proof}
From the DiPerna-Lions theory Theorem \ref{theorem - DiPLi} it is sufficient to demonstrate that the vector field $b$ has the renormalization property.

Let $u\in L^{\infty}\left(\left(-\infty,T\right)\times\R^{n}\right)$ be a weak solution of \eqref{transport equation} on $\left(-\infty,T\right)\times\R^{n}$ where the vector field $b$ is extended by zero for negative time.
Let $\beta\in C^{1}\left(\R\right)$ and $\phi\in C_{c}^{\infty}\left(\left(-\infty,T\right)\times\R^{n}\right)$. Fix $\varepsilon$ in the range $0<\varepsilon<1$ and define
\begin{equation}\label{def:test-function}
\phi_{\varepsilon}\left(t,x\right)\define \phi\left(t,x\right)\chi_{0}\left(d_{S}\left(t,x\right)/\varepsilon\right)
\end{equation}
where $\chi_{0}\in C^{\infty}\left(\R\right)$ satisfies
\[
\chi_{0}\left(z\right)
=
\begin{cases}
0 & \abs{z}\leq \frac{1}{2} \\
1 & \abs{z}\geq 1.
\end{cases}
\]
Observe that $\phi_{\varepsilon}$ is Lipschitz (although not necessarily smooth). An unpublished result of Serrin \cite{SerrinUnpublished} (see also \cite{LeoniMorini07}, \cite{Stampacchia65}, \cite{BoccardoMurat82}, and \cite{MarcusMizel72}) ensures that the chain rule applies almost everywhere for the composition of the Lipschitz functions $\chi_{0}$ and $d_{S}$, hence
\begin{align}
\partial_{t} \phi_{\varepsilon} 
&= \phi \chi_{0}^{\prime}\left(d_{S}/\varepsilon\right)\frac{1}{\varepsilon} \partial_{t} d_{S} + \chi_{0}\left(d_{S}/\varepsilon\right) \partial_{t} \phi \label{time derivative mollifier}
\intertext{and}
\nabla \phi_{\varepsilon} &= \phi \chi_{0}^{\prime}\left(d_{S}/\varepsilon\right)\frac{1}{\varepsilon} \nabla d_{S} + \chi_{0}\left(d_{S}/\varepsilon\right) \nabla \phi \label{space derivative mollifier}
\end{align}
almost everywhere on $\left(-\infty,T\right)\times\R^{n}$.

For brevity we adopt the notation
\[
\set{d_{S}>\varepsilon}\define \set{\left(t,x\right)\in \left(-\infty,T\right)\times\R^{n} \st  d_{S}\left(t,x\right)>\varepsilon},
\]
and similarly define the sets $\set{d_{S}\geq \varepsilon}$ and $\set{d_{S}<\varepsilon}$.

Let
\begin{align*}
I\left(\varepsilon\right) \define \iint\limits_{\set{d_{S}>\varepsilon}} \beta\left(u\right)\left(\frac{\partial \phi}{\partial t} + b\cdot\nabla \phi + \divergence b\;\phi\right)\;\rmd x\;\rmd t
\intertext{so, as $\mu_{n+1}\left(S\right)=0$,}
\lim_{\varepsilon \rightarrow 0} I\left(\varepsilon\right) = \int\limits_{-\infty}^{T}\int_{\R^{n}} \beta\left(u\right)\left(\frac{\partial\phi}{\partial t} + b\cdot\nabla \phi + \divergence b\;\phi\right)\;\rmd x\;\rmd t.
\end{align*}

Now, as $\phi\left(t,x\right) = \phi_{\varepsilon}\left(t,x\right)$ on $\set{d_{S}>\varepsilon}$ it follows that
\begin{align}
I\left(\varepsilon\right) =& \iint\limits_{\set{d_{S}>\varepsilon}} \beta\left(u\right)\left(\frac{\partial\phi_{\varepsilon}}{\partial t} + b\cdot\nabla \phi_{\varepsilon} + \divergence b\;\phi_{\varepsilon}\right)\;\rmd x\;\rmd t \notag \\
=& \iint\limits_{\set{d_{S} > \varepsilon/4}} \beta\left(u\right)\left(\frac{\partial\phi_{\varepsilon}}{\partial t} + b\cdot\nabla \phi_{\varepsilon} + \divergence b\;\phi_{\varepsilon}\right)\;\rmd x\;\rmd t \label{nice region} \\
&-\iint\limits_{\set{\varepsilon \geq d_{S} > \varepsilon/4}} \beta\left(u\right)\left(\frac{\partial\phi_{\varepsilon}}{\partial t} + b\cdot\nabla \phi_{\varepsilon} + \divergence b\;\phi_{\varepsilon}\right)\;\rmd x\;\rmd t. \notag 
\intertext{By assumption \ref{b renormalized on complement} and Theorem 1.9 the vector field $b$ has the local renormalization property on $\Omega=\set{d_{S}>\varepsilon/4}$. Consequently, as $\support \phi_{\varepsilon} \subset \subset \set{d_{S}>\varepsilon/4}$ and $\phi_{\varepsilon}$ is Lipschitz it follows by Remark \ref{remark - Lipschitz test functions} that that the integral \eqref{nice region} vanishes. }
I\left(\varepsilon\right) =& -\iint\limits_{\set{\varepsilon \geq d_{S} > \varepsilon/4}} \beta\left(u\right)\left(\frac{\partial\phi_{\varepsilon}}{\partial t} + b\cdot\nabla \phi_{\varepsilon} + \divergence b\;\phi_{\varepsilon}\right)\;\rmd x\;\rmd t \notag\\
=& -\iint\limits_{\set{ d_{S} \leq \varepsilon}} \beta\left(u\right)\left(\frac{\partial\phi_{\varepsilon}}{\partial t} + b\cdot\nabla \phi_{\varepsilon} + \divergence b\;\phi_{\varepsilon}\right)\;\rmd x\;\rmd t \label{key integral}
\end{align}
as $\phi_{\varepsilon}$ vanishes on $\set{d_{S}\leq \varepsilon/4}$.

It remains to demonstrate that \eqref{key integral} vanishes at the limit: as $\abs{\phi_{\varepsilon}} \leq \abs{\phi}$ and $\divergence b\in L^{1}\left(\left(0,T\right)\times \R^{n}\right)$ we immediately obtain
\[
\lim_{\varepsilon\rightarrow 0}\iint\limits_{\set{d_{S} \leq \varepsilon}} \beta\left(u\right) \divergence b\  \phi_{\varepsilon}\; \rmd x\; \rmd t = 0
\]
so it is sufficient to demonstrate that the limit of
\[
J\left(\varepsilon\right) \define \iint\limits_{\set{d_{S} \leq \varepsilon}} \abs{\beta\left(u\right)\left(\frac{\partial\phi_{\varepsilon}}{\partial t} + b\cdot\nabla \phi_{\varepsilon}\right)}\;\rmd x\;\rmd t
\]
is zero.

Using the identities \eqref{time derivative mollifier} and \eqref{space derivative mollifier},
\begin{align}
J\left(\varepsilon\right)= &\iint_{\set{d_{S}\leq \varepsilon}}
\Big\lvert\beta\left(u\right)\phi \chi_{0}^{\prime}\left(d_{S}/\varepsilon\right)\frac{1}{\varepsilon} \partial_{t} d_{S}
+ \beta\left(u\right)\chi_{0}\left(d_{S}/\varepsilon\right) \partial_{t} \phi \notag\\
&\qquad\qquad+\beta\left(u\right) \phi \chi_{0}^{\prime}\left(d_{S}/\varepsilon\right)\frac{1}{\varepsilon} b\cdot\nabla d_{S}
+ \beta\left(u\right) \chi_{0}\left(d_{S}/\varepsilon\right) b\cdot\nabla \phi \Big\rvert \;\rmd x\;\rmd t.\notag \\
\leq &\iint_{\set{d_{S}\leq\varepsilon}} \frac{1}{\varepsilon}\abs{ \beta\left(u\right)\phi \chi_{0}^{\prime}\left(d_{S}/\varepsilon\right)\left(\partial_{t} d_{S} + b\cdot \nabla d_{S}\right)}\;\rmd x\;\rmd t\label{awkward}\\
&+\iint_{\set{d_{S}\leq\varepsilon}} \abs{\beta\left(u\right)\chi_{0}\left(d_{S}/\varepsilon\right) \left(\partial_{t} \phi + b\cdot\nabla \phi\right)}\;\rmd x\;\rmd t.\label{easy}
\end{align}
As $\chi_{0}$ and $\beta\left(u\right)$ are bounded, and $b\in L^{1}_{\loc}\left(\left(0,1\right)\times \R^{n}\right)$ the integral \eqref{easy} vanishes as $\varepsilon\rightarrow 0$. Writing $J_{1}\left(\varepsilon\right)$ for the integral \eqref{awkward} we see that
\begin{align*}
J_{1}\left(\varepsilon\right) &\leq C \frac{1}{\varepsilon} \iint_{\set{d_{S}\leq \varepsilon}} 1 + \abs{b\cdot \nabla d_{S}}\;\rmd x\;\rmd t
\end{align*}
as $\beta\left(u\right)$, $\phi$ and $\chi_{0}^{\prime}$ are bounded, and $\abs{\partial_{t} d_{S}}\leq 1$ as distance functions have Lipschitz constant $1$. Applying H\"older's inequality with exponents satisfying \ref{distance condition} we obtain
\begin{align*}
J_{1}\left(\varepsilon\right) \leq & C\frac{1}{\varepsilon} \norm{\left(1+\abs{b\cdot\nabla d_{S}}\right)\vert_{\set{d_{S}\leq \varepsilon}}}_{L^{p}\left(-1,T;L^{q}\left(\R^{n}\right)\right)} \left( \int_{-1}^{T} \left( \int_{\set{x\in \R^{n} \st d_{S}\left(t,x\right)\leq\varepsilon}} 1 \;\rmd x\right)^{\frac{p^{*}}{q^{*}}} \rmd t\right)^{\frac{1}{p^{*}}}.
\end{align*}

Now, by Chebyshev's inequality
\begin{align}
&\frac{1}{\varepsilon}\left( \int_{-1}^{T} \mu_{n}\left(\set{x\in \R^{n} \st d_{S}\left(t,x\right)^{-1}>1/\varepsilon}\right)^{\frac{p^{*}}{q^{*}}} \rmd t\right)^{\frac{1}{p^{*}}}\label{weak chebyshev}\\
&\leq \frac{1}{\varepsilon} \left(\int_{-1}^{T} \left(\varepsilon^{q^{*}}\int\limits_{\set{x \in \R^{n} \st d_{S}(t,x) \leq \varepsilon}} d_{S}(t,x)^{-q^{*}}\;\rmd x\right)^{\frac{p^{*}}{q^{*}}}\; \rmd t \right)^{\frac{1}{p^{*}}}\notag\\
&= \norm{d_{S}^{-1}\vert_{\set{d_{S}\leq \varepsilon}}}_{L^{p^{*}}\left(-1,T;L^{q^{*}}\left(\R^{n}\right)\right)}.\notag
\end{align}

Hence
\begin{align*}
J_{1}\left(\varepsilon\right) \leq C \norm{\left(1+\abs{b\cdot\nabla d_{S}}\right)\vert_{\set{d_{S}\leq \varepsilon}}}_{L^{p}\left(-1,T;L^{q}\left(\R^{n}\right)\right)} \norm{d_{S}^{-1}\vert_{\set{d_{S}\leq \varepsilon}}}_{L^{p^{*}}\left(-1,T;L^{q^{*}}\left(\R^{n}\right)\right)}.
\end{align*}
This tends to zero as $\varepsilon\rightarrow 0$ as $1+\abs{b\cdot\nabla d_{S}}\in L^{p}\left(-1,T;L^{q}_{\loc}\left(\R^{n}\right)\right)$ from \eqref{b dot distance} and $d_{S}^{-1}\in L^{p^{*}}\left(-1,T;L^{q^{*}}_{\loc}\left(\R^{n}\right)\right)$ from \eqref{distance reciprocal}. Consequently,
\[
\lim_{\varepsilon\rightarrow 0} I\left(\varepsilon\right) =
\int\limits_{-\infty}^{T}\int_{\R^{n}} \beta\left(u\right)\left(\frac{\partial\phi}{\partial t} + b\cdot\nabla \phi + \divergence b\;\phi\right)\;\rmd x\;\rmd t = 0
\]
As $\beta\in C^{1}\left(\R\right)$ was an arbitrary map and $u\in L^{\infty}\left(\left(0,T\right)\times\R^{n}\right)$ an arbitrary weak solution it follows that the vector field $b$ has the renormalization property on $\left(-\infty,T\right)\times\R^{n}$.
\end{proof}

Remarks:
\begin{itemize}
\item From \eqref{weak chebyshev} it is sufficient for the spatial component of $d_{S}^{-1}$ to be locally weak-$L^{q^{*}}\left(\R^{n}\right)$.
\item It is straightforward (but notationally demanding) to adapt the above proof for unbounded weak solutions $u \in L^{p}\left(0,T;L^{q}\left(\R^{n}\right)\right)$.
\end{itemize}

\subsection{Avoidance of Singularities}
We now show that almost every trajectory of the regular Lagrangian Flow does not intersect the set $S$, namely the second part of Theorem \ref{theorem - all results}.
\begin{theorem}\label{theorem - new avoidance}
Let $S\subset\left[0,T\right]\times\R^{n}$ be compact, and suppose that the assumption \ref{distance condition} of Theorem \ref{theorem - existence and uniqueness} is satisfied. If $X$ is a regular Lagrangian flow solution of \eqref{ODE} then $X$ avoids the set $S$.
\end{theorem}

\begin{proof}
Let $\Omega = \set{x\in \R^{n} \st X\left(t,x\right) \in S\ \text{for some}\; t \in \left(0,T\right]}$. For each $r_{0}>0$ and $0<\delta<r_{0}$ define
\[
F\left(\delta\right) = \set{ x\in \Omega \st  d_{S}(0,x) \geq r_{0},\ \tau_{\delta}(x)<T}
\]
where
\[
\tau_{\delta}\left(x\right) \define
\begin{cases}
\sup \set{ t^\ast \st d_{S}\left(t,X\left(t,x\right)\right) \geq \delta\quad \forall t\in\left[0,t^\ast\right]} &\text{if } d_{S}(0,x)> \delta \\
0 &\text{if } d_{S}(0,x) \leq \delta.
\end{cases}
\]
Following Aizenman \cite{Aizenman78a} and Robinson \& Sharples \cite{RobinsonSharples13JDE} it is sufficient to show that $\mu_{n}\left(F
\left(\delta\right)\right) \rightarrow 0$ as $\delta\rightarrow 0$.

Define the Lipschitz function
\[
g\left(y\right) = \begin{cases}
\log\left(r_{0}/y\right) & \delta \leq y \leq r_{0} \\
0 &r_{0}<y
\end{cases}
\]
and note that
\begin{align}
g\left(d_{S}(0,x)\right) &= 0 &\forall x&\in F\left(\delta\right)\label{g zero}\\
g\left(d_{S}\left(\tau_{\delta}(x),X\left(\tau_{\delta}(x),x\right)\right)\right) &= g\left(\delta\right) &\text{a.e.}\ x&\in F\left(\delta\right)\label{g delta}
\end{align}
as the trajectories $t \mapsto X\left(t,x\right)$ are continuous for almost every $x\in \R^{n}$.

Now $d_{S}$ is Lipschitz with Lipschitz constant $1$ so by Rademacher's Theorem there is a set $N$ with $\mu_{n+1}(N)=0$ such that the derivatives $\frac{\partial }{\partial t} d_{S}\left(t,x\right)$ and $\nabla d_{S}\left(t,x\right)$ exist, and are bounded by $1$, for all $\left(t,x\right)\notin N$. The compressibility constant \eqref{compressibility constant} then ensures that $\frac{\partial }{\partial t} d_{S}\left(t,X\left(t,x\right)\right)$ and $\nabla d_{S}\left(t,X\left(t,x\right)\right)$ exist, and are bounded by $1$, for almost every $\left(t,x\right) \in \left[0,T\right] \times \R^{n}$.

Further, as trajectories are absolutely continuous, for almost every $x\in \R^{n}$
\[
t\mapsto d_{S}\left(t,X\left(t,x\right)\right)
\]
is absolutely continuous, hence is differentiable for almost every $t\in\left[0,T\right]$. It follows from Marcus \& Mizel \cite{MarcusMizel72} that for almost every $x\in \R^{n}$
\begin{align}
\frac{\rmd }{\rmd t} d_{S}\left(t,X\left(t,x\right)\right) &= \frac{\partial}{\partial t} d_{S} \left(t,X\left(t,x\right)\right) + \frac{\partial X}{\partial t} \cdot \nabla d_{S}\left(t,X\left(t,x\right)\right)\quad\text{a.e.}\  t\in\left[0,T\right]\notag\\
&= \frac{\partial}{\partial t} d_{S} \left(t,X\left(t,x\right)\right) + b\left(t,X\left(t,x\right)\right)\cdot \nabla d_{S}\left(t,X\left(t,x\right)\right)\quad\text{a.e.}\  t\in\left[0,T\right]\notag
\intertext{hence for almost every $x\in\R^{n}$, for almost every $t\in\left[0,T\right]$}
\abs{\frac{\rmd }{\rmd t} d_{S}\left(t,X\left(t,x\right)\right)} &\leq 1 + \abs{b\left(t,X\left(t,x\right)\right)\cdot \nabla d_{S}\left(t,X\left(t,x\right)\right) }.\label{distance derivative}
\end{align}

Next, as $g$ is Lipschitz, it follows from Serrin \& Varberg \cite{SerrinVarberg69} that for almost every $x\in \R^{n}$, for almost every $t\in \left[0,T\right]$
\begin{align}
\frac{\rmd }{\rmd t} g\left(d_{S}\left(t,X\left(t,x\right)\right)\right) &= g^{\prime}\left(d_{S}\left(t,X\left(t,x\right)\right)\right) \frac{\rmd}{\rmd t} d_{S}\left(t,X\left(t,x\right)\right).\label{g distance derivative}
\end{align}

Now, from \eqref{g zero} and \eqref{g delta}
\begin{align*}
\mu_{n}\left(F\left(\delta\right)\right) \abs{g\left(\delta\right)} &= \int_{F\left(\delta\right)} \abs{g\left(d_{S}\left(\tau_{\delta}(x),X\left(\tau_{\delta}(x),x\right)\right)\right) - g\left(d_{S}\left(0,x\right)\right)}\ \rmd x \\
&= \int_{F\left(\delta\right)} \abs{\int_{0}^{\tau_{\delta}(x)} \frac{\rmd }{\rmd t} g\left(d_{S}\left(t,X\left(t,x\right)\right)\right) \  \rmd t}\ \rmd x \\
\intertext{from \eqref{distance derivative} and \eqref{g distance derivative}}
&\leq \int_{F\left(\delta\right)} \int_{0}^{\tau_{\delta}(x)} \abs{ g^{\prime}\left(d_{S}\left(t,X\left(t,x\right)\right)\right)} \left(1 + \abs{b\left(t,X\left(t,x\right)\right) \cdot \nabla d_{S}\left(t,X\left(t,x\right)\right)}\right)\ \rmd t\ \rmd x.
\intertext{As the integrand is measurable from $\left[0,T\right]\times \R^{n}\rightarrow \R$ we apply Fubini's Theorem}
&= \int_{0}^{T} \int_{\R^{n}} \abs{ g^{\prime}\left(d_{S}\left(t,X\left(t,x\right)\right)\right)} \left(1 + \abs{b\left(t,X\left(t,x\right)\right) \cdot \nabla d_{S}\left(t,X\left(t,x\right)\right)}\right)\ \rmd x\ \rmd t\\
& \leq L \int_{0}^{T} \int_{\R^{n}} \abs{ g^{\prime}\left(d_{S}\left(t,x\right)\right)} \left(1 + \abs{b\left(t,x\right) \cdot \nabla d_{S}\left(t,x\right)}\right)\ \rmd x\ \rmd t.
\end{align*}
Finally, as
\[
g^{\prime}\left(y\right) = \begin{cases}
-\frac{1}{y} & \delta<y<r_{0} \\
0 & y>r_{0}
\end{cases}.
\]
we conclude that
\begin{align*}
\mu_{n}\left(F\left(\delta\right)\right)\abs{g\left(\delta\right)} &\leq
 L \int_{0}^{T} \int\limits_{\set{x \st d_{S}(t,x)<r_{0}}} d_{S}\left(t,x\right)^{-1} \left(1 + \abs{b\left(t,x\right) \cdot \nabla d_{S}\left(t,x\right)}\right)\ \rmd x\ \rmd t
\end{align*}
which is finite from H\"older's inequality and the conditions \eqref{b dot distance} and \eqref{distance condition}. As $\abs{g\left(\delta\right)} = \log\left(r_{0}/\delta\right) \rightarrow 0$ as $\delta\rightarrow 0$, it follows that $\mu_{n}\left(F\left(\delta\right)\right)\rightarrow 0$ as required.
\end{proof}

\subsection{Codimension print of trajectories}\label{section - codimension of trajectories}
In many applications we wish to consider the anisotropic detail of the graph of trajectories
\[
S\define \set{\left(t,Z\left(t,x\right)\right) \st t\in\left[0,T\right], x\in S_{0}}\subset \left[0,T\right]\times\R^{n}
\]
where $S_{0}\subset \R^{n}$ is some set of initial data. For sufficiently regular maps $Z$ this graph will have similar anisotropic detail to the product $\left[0,T\right]\times S_{0}$ in which case the codimension print would be immediately given by Theorem \ref{theorem - codimension product print}.

If the map $\left(t,x\right)\mapsto \left(t,Z\left(t,x\right)\right)$ is bi-Lipschitz then it is not difficult to show that
\[
d_{\left[0,T\right]\times S_{0}}\left(t,x\right)\leq C d_{S}\left(t,Z\left(t,x\right)\right) \leq C^{2} d_{\left[0,T\right]\times S_{0}}\left(t,x\right)
\]
from which it follows that the codimension prints of $S$ and $\left[0,T\right]\times S_{0}$ are identical, in which case the codimension print of $S$ is immediately given by Theorem \ref{theorem - codimension product print}.

In general the situation is more complicated: first, if an individual trajectory is not Lipschitz in time then the graph it traces can have a large fractal dimension. Secondly, if $Z$ is not bi-Lipschitz in space then the box-counting dimension of the temporal section $S\left(t\right)$ may vary in time.

For example if the trajectories are described by $Z\left(t,x\right)\define x+t\left(x^{2}-x\right)$ then the set of initial data $S\left(0\right)\define \set{n^{-1}\st n\in \mathbb{N}}$ evolves to the set $S\left(1\right)=\set{Z\left(1,x\right) \st x\in S\left(0\right)} = \set{n^{-2} \st n\in\mathbb{N}}$, in which case the upper and lower box-counting dimensions are not preserved as
\[
\dim_{LB} S\left(0\right)=\dim_{B}S\left(0\right)=\frac{1}{2} > \frac{1}{3} = \dim_{LB}S\left(1\right)=\dim_{B}S\left(1\right)
\]
(see Example 13.4 of Robinson \cite{BkRobinson01}).

However, if the trajectories have some uniform H\"older regularity in time then we can describe the codimension print in terms of the H\"older exponents and the maximum box-counting dimension of the temporal sections. This result requires no spatial regularity of the map $Z$.

\begin{theorem}\label{theorem - singular trajectories}
Let $S_{0}\subset\R^{n}$ be bounded. Suppose that for some $\alpha$ in the range $0<\alpha\leq 1$ the map $Z\colon\left[0,T\right]\times S_{0}\rightarrow \R^{n}$ is $\alpha$-H\"older continuous in $t$ uniformly in $x$, which is to say that there exists a $K>0$ with
\begin{align}\label{uniform Holder condition}
\abs{Z\left(t_{1},x\right)-Z\left(t_{2},x\right)}&\leq K \abs{t_{1}-t_{2}}^{\alpha} &\forall t_{1},t_{2}\in\left[0,T\right]\quad\forall x\in S_{0},
\end{align}
then the distance function $d_{S}$ of the set
\begin{align}\label{S family of curves}
S\define \set{\left(t,Z\left(t,x\right)\right) \st t\in\left[0,T\right], x\in S_{0}}\subset \left[0,T\right]\times\R^{n} 
\end{align}
satisfies
\[
d_{S}^{-1}\in L^{\infty}\left(0,T;L^{r}_{\loc}\left(\R^{n}\right)\right),
\]
(that is $\left(r,\infty\right) \in \print(S)$), for all $r<\alpha\left(n-\sup_{t\in\left[0,T\right]} \dim_{B}S\left(t\right)\right)$.
\end{theorem}

It is clear that $d_{S}\left(t,x\right)\leq d_{S\left(t\right)}\left(x\right)$ as for all $y\in S\left(t\right)$ the point $\left(t,y\right)\in S$ and so $d_{S}\left(t,x\right)\leq \abs{\left(x,t\right)-\left(y,t\right)}=\abs{x-y}$. In the following proof we see that the H\"older condition \eqref{uniform Holder condition} ensures that the converse inequality $d_{S\left(t\right)}\left(x\right)\leq \left(K+1\right) d_{S}\left(t,x\right)^{\alpha}$ holds.

\begin{proof}
From Lemma \ref{lemma - basic codimension properties} the codimension print is invariant under closure of sets so we can assume that $S$ is closed. Let $r<\alpha\left(n-\sup_{t\in\left[0,T\right]} \dim_{B}S\left(t\right)\right)$ and let $\delta>0$ be sufficiently small that $r+\delta<\alpha\left(n-\sup_{t\in\left[0,T\right]} \dim_{B}S\left(t\right)\right)$.

To show that $d_{S}^{-1}\in L^{\infty}\left(0,T;L^{r}_{\loc}\left(\R^{n}\right)\right)$ it is sufficient to demonstrate that
\begin{align}
\esssup_{t\in\left[0,T\right]} \int\limits_{\set{x\st d_{S}\left(t,x\right)<1}} d_{S}\left(t,x\right)^{-r} \rmd x \label{distance integral}
\end{align}
is finite, as $d_{S}^{-1}$ is bounded away from $S$.

Fix $\left(t,x\right)\in \left[0,T\right]\times\R^{n}$ such that $d_{S}\left(t,x\right)<1$ and let $\left(s,y\right)\in S$ be such that $d_{S}\left(t,x\right)= \abs{\left(t,x\right)-\left(s,y\right)}$. Now, $y=Z\left(s,y_{0}\right)$ for some $y_{0}\in S_{0}$, as $S$ has the form \eqref{S family of curves}, so certainly the point $Z\left(t,y_{0}\right)\in S\left(t\right)$. Consequently,
\begin{align}
d_{S\left(t\right)}\left(x\right) &\leq \abs{x-Z\left(t,y_{0}\right)} \notag
\leq  \abs{x-Z\left(s,y_{0}\right)} + \abs{Z\left(s,y_{0}\right)-Z\left(t,y_{0}\right)}\notag
\intertext{which, from the uniform H\"older condition \eqref{uniform Holder condition},}
&\leq \abs{x-y} + K\abs{t-s}^{\alpha} \leq d_{S}\left(t,x\right)+ K d_{S}\left(t,x\right)^{\alpha}\notag\\
&\leq \left(K+1\right) d_{S}\left(t,x\right)^{\alpha}\label{distance function slice bound} 
\end{align}
as $d_{S}\left(t,x\right) < 1$ and $\alpha \leq 1$.
This inequality yields the inclusion
\begin{align*}
 \set{x \st  d_{S}\left(t,x\right)< 1} &\subset \set{x \st  d_{S\left(t\right)}\left(x\right)< K+1} &\forall t\in\left[0,T\right]
\end{align*}
and the inequality $d_{S}\left(t,x\right)^{-1}\leq \left(\tfrac{1}{K+1}d_{S\left(t\right)}\left(x\right)\right)^{-1/\alpha}$ so for all $t\in\left[0,T\right]$
\begin{align}
I\left(t\right)\define \int\limits_{\set{x\st d_{S}\left(t,x\right)<1}} d_{S}\left(t,x\right)^{-r} \rmd x
\leq \int\limits_{\set{x\st d_{S\left(t\right)}\left(x\right)<K+1}} d_{S\left(t\right)}\left(x\right)^{-r/\alpha} \rmd x. \label{distance integral bounded by slice}
\end{align}

We write $M\define \left(K+1\right)$ and, following the argument of \cite{Aizenman78a}, we rewrite \eqref{distance integral bounded by slice} as
\begin{align*}
I\left(t\right)&\leq \int\limits_{\set{x\st d_{S\left(t\right)}\left(x\right)<M}} M^{-r/\alpha}\rmd x + \int\limits_{\set{x\st d_{S\left(t\right)}\left(x\right)<M}} \left(d_{S\left(t\right)}\left(x\right)^{-r/\alpha}-M^{-r/\alpha}\right) \rmd x \notag\\
&= M^{-r/\alpha}\mu_{n}\left(\set{x\st d_{S\left(t\right)}\left(x\right)<M}\right) + \int\limits_{\set{x\st d_{S\left(t\right)}\left(x\right)<M}} \int_{M^{-r/\alpha}}^{d_{S\left(t\right)}(x)^{-r/\alpha}} 1 \rmd s\;\rmd x\notag
\intertext{which, from Fubini's theorem,}
&= M^{-r/\alpha}\mu_{n}\left(\set{x\st d_{S\left(t\right)}\left(x\right)<M}\right) + \int_{M^{-r/\alpha}}^{\infty} \mu_{n}\left(\set{x \st  d_{S\left(t\right)}\left(x\right)<s^{-\alpha/r}}\right) \rmd s. \label{integral neighbourhood bound}
\end{align*}
Next, the diameter of the temporal sections $\omega\define \sup_{t\in\left[0,T\right]} \diameter S\left(t\right)<\infty$ as $S$ is bounded, so from Lemma \ref{lemma - Minkowski Sausage} there exists a constant $C$ dependent on $r/\alpha+\delta/\alpha$ and $\omega$ such that
\begin{align*}
\mu_{n}\left(\set{d_{S\left(t\right)}<\varepsilon}\right)&\leq C\varepsilon^{r/\alpha+\delta/\alpha} &\forall \varepsilon\in\left(0,M\right]\quad\forall t\in\left[0,T\right].
\end{align*}
Consequently,
\begin{align*}
I\left(t\right)\leq M^{-r/\alpha}CM^{r/\alpha+\delta/\alpha} + \int_{M^{-r/\alpha}}^{\infty} C \left(s^{-\alpha/r}\right)^{r/\alpha +\delta/\alpha} \rmd s,
\end{align*}
which is finite as $\left(\alpha/r\right)\left(r/\alpha +\delta/\alpha\right)= 1+\delta/r > 1$. Further, this bound is independent of $t$ so
we conclude that \eqref{distance integral} is finite from which it follows that $d_{S}^{-1} \in L^{\infty}\left(0,T;L^{r}_{\loc}\left(\R^{n}\right)\right)$, as required.
\end{proof}

\subsection{Application to the vortex-wave system}
\label{subsec:vortex-wave}

The purpose of this paragraph to explore the link between our main result stated in Theorem \ref{theorem - all results} and a situation arising in fluid dynamics to describe the interaction of a 2-dimensional fluid with several point vortices.
The resulting system, called vortex-wave system, is a coupling of a nonlinear transport equation for the vorticity of the fluid and a system of ODE for the evolution of the point vortices. It is derived from the incompressible Euler equations and was introduced by Marchioro and Pulvirenti \cite{Marpul, MarPul1} and 
by Starovo\v{i}tov \cite{starovoitov-existence, starovoitov-unique}.
 In the case of one single point vortex,  the vortex-wave system may be written as
\begin{equation}
\label{eq:vortex-wave}
\begin{cases}
\displaystyle \partial_t \omega+b \cdot \nabla \omega =0\\
\displaystyle v=\frac{x^\perp}{|x|^2}\ast \omega,\quad b=v+ \frac{(x-z(t))^\perp}{|x-z(t)|^2}\\
\displaystyle \dot{z}(t)=v(t,z(t)),
\end{cases}
\end{equation} 
with $\omega\in L^\infty((0,T), L^1\cap L^p(\R^2))$ for some $p>2$, denoting the vorticity, and $z\in W^{1,\infty}((0,T))$ denoting the vortex trajectory. Note that $p>2$ indeed ensures that the velocity field generated by $\omega$, $v=\frac{x^\perp}{|x|^2}\ast \omega$, is uniformly bounded on $(0,T)\times \R^2$. In Crippa et al. \cite{CrippaLopesFilhoMiotNussenzveigLopes16}, it is proved that for any such weak solution $(\omega,z)$ to the system \eqref{eq:vortex-wave}, there exists a unique regular Lagrangian flow such that $\omega$ is constant along the flow trajectories. Moreover, generically, its trajectories do not collide with the trajectory of the point vortex (namely $X(t,x)$ avoids the set $S(t)=\{z(t)\}$ for almost every $x\in \R^2$). Actually, this property is proved in \cite{CrippaLopesFilhoMiotNussenzveigLopes16} for more general vector fields $b$ of the form above, where 
 $v$ is a given bounded vector field satisfying assumptions that are essentially the ones required by Ambrosio's result. In particular, the point vortex trajectory $t\mapsto z(t)$ is Lipschitz. 

When there is only one point vortex trajectory, the singular set is defined by
\begin{equation}
S=\left\{(t,z(t));\quad t\in [0,T] \right\}
\end{equation}
and its temporal sections are $S(t)=\left\{z(t) \right\}$, so that $\text{dim}_B(S(t))=0$ for all $t\in [0,T]$, and $d_{S(t)}(x)=|x-z(t)|$. 

Moreover, since $z$ is Lipschitz (with Lipschitz constant given $K$), we have
\begin{equation}\label{ineq:compar}
\frac{1}{K+1} d_{S(t)}(x)\leq d_S(t,x)\leq d_{S(t)}(x)\leq (K+1)d_S(t,x).
\end{equation}
So by Theorem \ref{theorem - singular trajectories} we retrieve that $d_S^{-1}\in L^\infty(0,T; L_{\loc}^{r}(\R^2))$ for all $r<2$. In order to apply Proposition \ref{prop:traj} yielding Theorem \ref{theorem - all results}, we still need to check that \ $b\cdot \nabla d_{S} \in L^{1}\left(0,T;L^{q}_{\loc}\left(\R^{2}\right)\right)$ for some $q>2$. However we do not know the explicit form of $d_S$ in the present setting; so, we need to come back to the explicit distance $d_{S(t)}(x)$.

In \cite{CrippaLopesFilhoMiotNussenzveigLopes16}, the renormalization property is obtained by considering test functions depending on the quantity $d_{S(t)}\left(x\right)/\varepsilon=|x-z(t)|/\varepsilon$. It is then based on the observation that $b\cdot \nabla d_{S(t)}(x)=v\cdot \nabla d_{S(t)}(x)$ is not singular. So, the first part of Theorem \ref{theorem - all results} (namely of Theorem \ref{theorem - existence and uniqueness}), which is based on test functions defined by \eqref{def:test-function}, is an extension of the method of \cite{CrippaLopesFilhoMiotNussenzveigLopes16} to more general singular fields and sets. Theorem \ref{theorem - new avoidance} is also an avoidance property applying to more general singular sets. We remark that it can be used to retrieve easily the  avoidance property established in \cite{CrippaLopesFilhoMiotNussenzveigLopes16}. Indeed, coming back to the proof of Theorem \ref{theorem - new avoidance}, we get by virtue of \eqref{ineq:compar}
\begin{equation}
\delta\leq d_{S(\tau_\delta)}(x)\leq (K+1)\delta,\quad \forall x\in F(\delta).
\end{equation}We can assume that $\delta$ is sufficienly small so that $(K+1)\delta \leq \delta^{1/2}$. Hence
$$g(\delta)= \log(1/\delta)\leq 2  \log(1/d_{S(\tau_\delta)}(x))=g(d_{S(\tau_\delta)}(x))-g(d_{S_0}(x)).$$
Hence we may mimick the subsequent computations, replacing $d_S(t, X(t,x))$ by $d_{S(t)}(X(t,x))$ and observing the cancellation on $b(t,\cdot)\cdot \nabla d_{S(t)}$. It follows that $\mu_2(F(\delta))\to 0$ as $\delta\to 0$. Note that this argument would apply equally to a point vortex trajectory only H\"older continuous in time. 

We finally mention that this applies to a finite number of point vortex trajectories that do not collide on $[0,T]$: this follows from a straightforward adaptation of the previous arguments.

\section{Conclusion}
We have demonstrated that the renormalization theory of DiPerna \& Lions and Ambrosio can be extended to vector fields that are $BV$ off a set of singularities, provided that the anisotropic fractal detail of the set of singularities is known and the component of the vector field normal to these singularities is sufficiently small. We provide a way of calculating the necessary anisotropic detail for a singular set composed of trajectories. The renormalization theory then gives the existence and uniqueness of solutions to the transport equation \eqref{transport equation}, and the corresponding ordinary differential equation \eqref{ODE}. Further, the trajectories of the flow solution avoid the singular set, which we demonstrated by improving upon the avoidance results of Aizenman and Robinson \& Sharples. We retrieve known results in point vortex dynamics in the particular case where the singular set is given by the graphs of a finite number of point vortex Lipschitz trajectories.

\section{Acknowledgments}

E. M. is partially supported by the french Agence Nationale de la Recherche through the following projects: SINGFLOWS (grant ANR-18-CE40-0027-01), and INFAMIE (grant ANR-15-CE40-01).

\bibliographystyle{plain}
\bibliography{biblio.bib}

\begin{thebibliography}{10}

\bibitem{Aizenman78a}
Michael Aizenman.
\newblock A sufficient condition for the avoidance of sets by measure
  preserving flows in {${\bf R}^{n}$}.
\newblock {\em Duke Math. J.}, 45(4):809--813, 1978.

\bibitem{Ambrosio04}
Luigi Ambrosio.
\newblock Transport equation and {C}auchy problem for {$BV$} vector fields.
\newblock {\em Invent. Math.}, 158(2):227--260, 2004.

\bibitem{Ambrosio08}
Luigi Ambrosio.
\newblock Transport equation and {C}auchy problem for non-smooth vector fields.
\newblock In {\em Calculus of variations and nonlinear partial differential
  equations}, volume 1927 of {\em Lecture Notes in Math.}, pages 1--41.
  Springer, Berlin, 2008.

\bibitem{BoccardoMurat82}
Lucio Boccardo and Fran\c{c}ois Murat.
\newblock Remarques sur l'homog\'{e}n\'{e}isation de certains probl\`emes
  quasi-lin\'{e}aires.
\newblock {\em Portugal. Math.}, 41(1-4):535--562 (1984), 1982.

\bibitem{Bouchut01}
Fran\c{c}ois Bouchut.
\newblock Renormalized solutions to the {V}lasov equation with coefficients of
  bounded variation.
\newblock {\em Arch. Ration. Mech. Anal.}, 157(1):75--90, 2001.

\bibitem{caprino}
Silvia Caprino, Carlo Marchioro, Evelyne Miot, and Mario Pulvirenti.
\newblock On the attractive plasma-charge system in 2-d.
\newblock {\em Comm. Partial Differential Equations}, 37(7):1237--1272, 2012.

\bibitem{CapuzzoDolcettaPerthame96}
Italo Capuzzo~Dolcetta and Beno\^it Perthame.
\newblock On some analogy between different approaches to first order {PDE}'s
  with nonsmooth coefficients.
\newblock {\em Adv. Math. Sci. Appl.}, 6(2):689--703, 1996.

\bibitem{CiprianoCruzeiro05}
Fernanda Cipriano and Ana~Bela Cruzeiro.
\newblock Flows associated with irregular {$\mathbb{R}^d$}-vector fields.
\newblock {\em J. Differential Equations}, 219(1):183--201, 2005.

\bibitem{colombinilerner05}
Ferruccio Colombini and Nicolas Lerner.
\newblock Uniqueness of {$L^\infty$} solutions for a class of conormal {$BV$}
  vector fields.
\newblock In {\em Geometric analysis of {PDE} and several complex variables},
  volume 368 of {\em Contemp. Math.}, pages 133--156. Amer. Math. Soc.,
  Providence, RI, 2005.

\bibitem{Crippa08}
Gianluca Crippa.
\newblock The ordinary differential equation with non-{L}ipschitz vector
  fields.
\newblock {\em Boll. Unione Mat. Ital. (9)}, 1(2):333--348, 2008.

\bibitem{CrippaDeLellis08}
Gianluca Crippa and Camillo De~Lellis.
\newblock Estimates and regularity results for the {D}i{P}erna-{L}ions flow.
\newblock {\em J. Reine Angew. Math.}, 616:15--46, 2008.

\bibitem{Ligabue}
Gianluca Crippa, Silvia Ligabue, and Chiara Saffirio.
\newblock Lagrangian solutions to the {V}lasov-{P}oisson system with a point
  charge.
\newblock {\em Kinet. Relat. Models}, 11(6):1277--1299, 2018.

\bibitem{CrippaLopesFilhoMiotNussenzveigLopes16}
Gianluca Crippa, Milton~C. Lopes~Filho, Evelyne Miot, and Helena~J.
  Nussenzveig~Lopes.
\newblock Flows of vector fields with point singularities and the vortex-wave
  system.
\newblock {\em Discrete Contin. Dyn. Syst.}, 36(5):2405--2417, 2016.

\bibitem{DeLellis08}
Camillo De~Lellis.
\newblock O{DE}s with {S}obolev coefficients: the {E}ulerian and the
  {L}agrangian approach.
\newblock {\em Discrete Contin. Dyn. Syst. Ser. S}, 1(3):405--426, 2008.

\bibitem{DeLellis08a}
Camillo De~Lellis.
\newblock Ordinary differential equations with rough coefficients and the
  renormalization theorem of {A}mbrosio [after {A}mbrosio, {D}i{P}erna,
  {L}ions].
\newblock Number 317, pages Exp. No. 972, viii, 175--203. 2008.
\newblock S\'{e}minaire Bourbaki. Vol. 2006/2007.

\bibitem{Depauw03}
Nicolas Depauw.
\newblock Non-unicit\'{e} du transport par un champ de vecteurs presque {BV}.
\newblock In {\em Seminaire: \'{E}quations aux {D}\'{e}riv\'{e}es {P}artielles,
  2002--2003}, S\'{e}min. \'{E}qu. D\'{e}riv. Partielles, pages Exp. No. XIX,
  9. \'{E}cole Polytech., Palaiseau, 2003.

\bibitem{DiPernaLions89a}
Ronald~J. DiPerna and Pierre-Louis Lions.
\newblock Ordinary differential equations, transport theory and {S}obolev
  spaces.
\newblock {\em Invent. Math.}, 98(3):511--547, 1989.

\bibitem{EvansGariepy92}
Lawrence~C. Evans and Ronald~F. Gariepy.
\newblock {\em Measure theory and fine properties of functions}.
\newblock Studies in Advanced Mathematics. CRC Press, Boca Raton, FL, 1992.

\bibitem{BkFalconer03}
Kenneth Falconer.
\newblock {\em Fractal geometry}.
\newblock John Wiley \& Sons, Inc., Hoboken, NJ, second edition, 2003.
\newblock Mathematical foundations and applications.

\bibitem{Fefferman}
Charles~L. Fefferman, Benjamin~C. Pooley, and Jos\'{e}~L. Rodrigo.
\newblock Non-conservation of dimension in divergence-free solutions of passive
  and active scalar systems.
\newblock {\em Arch. Ration. Mech. Anal.}, 242(3):1445--1478, 2021.

\bibitem{LeoniMorini07}
Giovanni Leoni and Massimiliano Morini.
\newblock Necessary and sufficient conditions for the chain rule in
  {$W^{1,1}_{\rm loc}(\mathbb{R}^N;\mathbb{R}^d)$} and {${\rm BV}_{\rm
  loc}(\mathbb{R}^N;\mathbb{R}^d)$}.
\newblock {\em J. Eur. Math. Soc. (JEMS)}, 9(2):219--252, 2007.

\bibitem{Lions98}
Pierre-Louis Lions.
\newblock Sur les \'{e}quations diff\'{e}rentielles ordinaires et les
  \'{e}quations de transport.
\newblock {\em C. R. Acad. Sci. Paris S\'{e}r. I Math.}, 326(7):833--838, 1998.

\bibitem{MarPul1}
Carlo Marchioro and Mario Pulvirenti.
\newblock On the vortex-wave system.
\newblock In {\em Mechanics, analysis and geometry: 200 years after
  {L}agrange}, North-Holland Delta Ser., pages 79--95. North-Holland,
  Amsterdam, 1991.

\bibitem{Marpul}
Carlo Marchioro and Mario Pulvirenti.
\newblock {\em Mathematical theory of incompressible nonviscous fluids},
  volume~96 of {\em Applied Mathematical Sciences}.
\newblock Springer-Verlag, New York, 1994.

\bibitem{MarcusMizel72}
Michael~B. Marcus and Victor~J. Mizel.
\newblock Absolute continuity on tracks and mappings of {S}obolev spaces.
\newblock {\em Arch. Rational Mech. Anal.}, 45:294--320, 1972.

\bibitem{Nguyen}
Quoc-Hung Nguyen.
\newblock Quantitative estimates for regular {L}agrangian flows with {$BV$}
  vector fields.
\newblock {\em Comm. Pure Appl. Math.}, 74(6):1129--1192, 2021.

\bibitem{BkRobinson01}
James~C. Robinson.
\newblock {\em Infinite-dimensional dynamical systems}.
\newblock Cambridge Texts in Applied Mathematics. Cambridge University Press,
  Cambridge, 2001.
\newblock An introduction to dissipative parabolic PDEs and the theory of
  global attractors.

\bibitem{RobinsonSadowski09}
James~C. Robinson and Witold Sadowski.
\newblock Almost-everywhere uniqueness of {L}agrangian trajectories for
  suitable weak solutions of the three-dimensional {N}avier-{S}tokes equations.
\newblock {\em Nonlinearity}, 22(9):2093--2099, 2009.

\bibitem{RobinsonSadowski09a}
James~C. Robinson and Witold Sadowski.
\newblock A criterion for uniqueness of {L}agrangian trajectories for weak
  solutions of the 3{D} {N}avier-{S}tokes equations.
\newblock {\em Comm. Math. Phys.}, 290(1):15--22, 2009.

\bibitem{RobinsonSadowskiSharples12}
James~C. Robinson, Witold Sadowski, and Nicholas Sharples.
\newblock {On the regularity of Lagrangian trajectories corresponding to
  suitable weak solutions of the Navier-Stokes equations}.
\newblock {\em {Procedia IUTAM}}, {2012}.

\bibitem{RobinsonSharples13RAEX}
James~C. Robinson and Nicholas Sharples.
\newblock Strict inequality in the box-counting dimension product formulas.
\newblock {\em Real Anal. Exchange}, 38(1):95--119, 2012/13.

\bibitem{RobinsonSharples13JDE}
James~C. Robinson and Nicholas Sharples.
\newblock Dimension prints and the avoidance of sets for flow solutions of
  non-autonomous ordinary differential equations.
\newblock {\em J. Differential Equations}, 254(10):4144--4167, 2013.

\bibitem{Rogers88}
C.~Ambrose Rogers.
\newblock Dimension prints.
\newblock {\em Mathematika}, 35(1):1--27, 1988.

\bibitem{SerrinVarberg69}
James Serrin and Dale~E. Varberg.
\newblock A general chain rule for derivatives and the change of variables
  formula for the {L}ebesgue integral.
\newblock {\em Amer. Math. Monthly}, 76:514--520, 1969.

\bibitem{SerrinUnpublished}
{Serrin, J}.
\newblock {Unpublished}.

\bibitem{Stampacchia65}
Guido Stampacchia.
\newblock Le probl\`eme de {D}irichlet pour les \'{e}quations elliptiques du
  second ordre \`a coefficients discontinus.
\newblock {\em Ann. Inst. Fourier (Grenoble)}, 15(fasc. 1):189--258, 1965.

\bibitem{starovoitov-existence}
V.~N. Starovo\u{\i}tov.
\newblock Solvability of a problem on the motion of concentrated vortices in an
  ideal fluid.
\newblock {\em Dinamika Sploshn. Sredy}, (85):118--136, 165, 1988.

\bibitem{starovoitov-unique}
V.~N. Starovo\u{\i}tov.
\newblock Uniqueness of the solution to the problem of the motion of a point
  vortex.
\newblock {\em Sibirsk. Mat. Zh.}, 35(3):696--701, v, 1994.

\end{thebibliography}
@article {Aizenman78a,
    AUTHOR = {Aizenman, Michael},
     TITLE = {A sufficient condition for the avoidance of sets by measure
              preserving flows in {${\bf R}^{n}$}},
   JOURNAL = {Duke Math. J.},
  FJOURNAL = {Duke Mathematical Journal},
    VOLUME = {45},
      YEAR = {1978},
    NUMBER = {4},
     PAGES = {809--813},
      ISSN = {0012-7094},
   MRCLASS = {58F25 (34A15 58F12)},
  MRNUMBER = {518107},
MRREVIEWER = {D. Newton},
       URL = {http://projecteuclid.org/euclid.dmj/1077313100},
}

@article {Ambrosio04,
    AUTHOR = {Ambrosio, Luigi},
     TITLE = {Transport equation and {C}auchy problem for {$BV$} vector
              fields},
   JOURNAL = {Invent. Math.},
  FJOURNAL = {Inventiones Mathematicae},
    VOLUME = {158},
      YEAR = {2004},
    NUMBER = {2},
     PAGES = {227--260},
      ISSN = {0020-9910},
   MRCLASS = {35K15 (34A12 35L65)},
  MRNUMBER = {2096794},
MRREVIEWER = {J. W. Jerome},
       DOI = {10.1007/s00222-004-0367-2},
       URL = {https://doi.org/10.1007/s00222-004-0367-2},
}

@incollection {Ambrosio08,
    AUTHOR = {Ambrosio, Luigi},
     TITLE = {Transport equation and {C}auchy problem for non-smooth vector
              fields},
 BOOKTITLE = {Calculus of variations and nonlinear partial differential
              equations},
    SERIES = {Lecture Notes in Math.},
    VOLUME = {1927},
     PAGES = {1--41},
 PUBLISHER = {Springer, Berlin},
      YEAR = {2008},
   MRCLASS = {35F10 (34A12 35B30 76B03)},
  MRNUMBER = {2408257},
MRREVIEWER = {Sergey Nikolaevich Alekseenko},
       DOI = {10.1007/978-3-540-75914-0\_1},
       URL = {https://doi.org/10.1007/978-3-540-75914-0_1},
}

@article {BoccardoMurat82,
    AUTHOR = {Boccardo, Lucio and Murat, Fran\c{c}ois},
     TITLE = {Remarques sur l'homog\'{e}n\'{e}isation de certains probl\`emes
              quasi-lin\'{e}aires},
   JOURNAL = {Portugal. Math.},
  FJOURNAL = {Portugaliae Mathematica},
    VOLUME = {41},
      YEAR = {1982},
    NUMBER = {1-4},
     PAGES = {535--562 (1984)},
      ISSN = {0032-5155},
   MRCLASS = {35B99 (35J60 35K55)},
  MRNUMBER = {766874},
MRREVIEWER = {Jan Soko\l owski},
}

@article {Bouchut01,
    AUTHOR = {Bouchut, Fran\c{c}ois},
     TITLE = {Renormalized solutions to the {V}lasov equation with
              coefficients of bounded variation},
   JOURNAL = {Arch. Ration. Mech. Anal.},
  FJOURNAL = {Archive for Rational Mechanics and Analysis},
    VOLUME = {157},
      YEAR = {2001},
    NUMBER = {1},
     PAGES = {75--90},
      ISSN = {0003-9527},
   MRCLASS = {35F05 (35D10 35F25)},
  MRNUMBER = {1822415},
MRREVIEWER = {R. Glassey},
       DOI = {10.1007/PL00004237},
       URL = {https://doi.org/10.1007/PL00004237},
}

@article {Caprino,
    AUTHOR = {Caprino, Silvia and Marchioro, Carlo and Miot, Evelyne and Pulvirenti, Mario},
     TITLE = {On the attractive plasma-charge system in 2-d},
   JOURNAL = {Comm. Partial Differential Equations},
  FJOURNAL = {Communications in Partial Differential Equations},
    VOLUME = {37},
      YEAR = {2012},
    NUMBER = {7},
     PAGES = {1237--1272},
      ISSN = {0360-5302},
   MRCLASS = {82D10 (35F20 35Q83)},
  MRNUMBER = {2942982},
MRREVIEWER = {Yuxi Zheng},
       DOI = {10.1080/03605302.2011.653032},
       URL = {https://doi.org/10.1080/03605302.2011.653032},
}

@article {CiprianoCruzeiro05,
    AUTHOR = {Cipriano, Fernanda and Cruzeiro, Ana Bela},
     TITLE = {Flows associated with irregular {$\mathbb{R}^d$}-vector fields},
   JOURNAL = {J. Differential Equations},
  FJOURNAL = {Journal of Differential Equations},
    VOLUME = {219},
      YEAR = {2005},
    NUMBER = {1},
     PAGES = {183--201},
      ISSN = {0022-0396},
   MRCLASS = {60H07 (37C10)},
  MRNUMBER = {2181034},
MRREVIEWER = {Shi Zan Fang},
       DOI = {10.1016/j.jde.2005.02.015},
       URL = {https://doi.org/10.1016/j.jde.2005.02.015},
}

@incollection {colombinilerner05,
    AUTHOR = {Colombini, Ferruccio and Lerner, Nicolas},
     TITLE = {Uniqueness of {$L^\infty$} solutions for a class of conormal
              {$BV$} vector fields},
 BOOKTITLE = {Geometric analysis of {PDE} and several complex variables},
    SERIES = {Contemp. Math.},
    VOLUME = {368},
     PAGES = {133--156},
 PUBLISHER = {Amer. Math. Soc., Providence, RI},
      YEAR = {2005},
   MRCLASS = {35F05 (26A45 34A12)},
  MRNUMBER = {2126467},
MRREVIEWER = {Toka Diagana},
       DOI = {10.1090/conm/368/06776},
       URL = {https://doi.org/10.1090/conm/368/06776},
}

@article {CrippaDeLellis08,
    AUTHOR = {Crippa, Gianluca and De Lellis, Camillo},
     TITLE = {Estimates and regularity results for the {D}i{P}erna-{L}ions
              flow},
   JOURNAL = {J. Reine Angew. Math.},
  FJOURNAL = {Journal f\"{u}r die Reine und Angewandte Mathematik. [Crelle's
              Journal]},
    VOLUME = {616},
      YEAR = {2008},
     PAGES = {15--46},
      ISSN = {0075-4102},
   MRCLASS = {34C11 (35L45 37C10)},
  MRNUMBER = {2369485},
MRREVIEWER = {Guido Schneider},
       DOI = {10.1515/CRELLE.2008.016},
       URL = {https://doi.org/10.1515/CRELLE.2008.016},
}

@article {Crippa08,
    AUTHOR = {Crippa, Gianluca},
     TITLE = {The ordinary differential equation with non-{L}ipschitz vector
              fields},
   JOURNAL = {Boll. Unione Mat. Ital. (9)},
  FJOURNAL = {Bollettino della Unione Matematica Italiana. Serie 9},
    VOLUME = {1},
      YEAR = {2008},
    NUMBER = {2},
     PAGES = {333--348},
      ISSN = {1972-6724},
   MRCLASS = {37C10 (34A36 34G20 46E35)},
  MRNUMBER = {2424297},
MRREVIEWER = {Guido Schneider},
}

@article {Ligabue,
    AUTHOR = {Crippa, Gianluca and Ligabue, Silvia and Saffirio, Chiara},
     TITLE = {Lagrangian solutions to the {V}lasov-{P}oisson system with a
              point charge},
   JOURNAL = {Kinet. Relat. Models},
  FJOURNAL = {Kinetic and Related Models},
    VOLUME = {11},
      YEAR = {2018},
    NUMBER = {6},
     PAGES = {1277--1299},
      ISSN = {1937-5093},
   MRCLASS = {35Q83 (58J45 82D10)},
  MRNUMBER = {3815144},
MRREVIEWER = {Jonathan Ben-Artzi},
       DOI = {10.3934/krm.2018050},
       URL = {https://doi.org/10.3934/krm.2018050},
}

@article {CrippaLopesFilhoMiotNussenzveigLopes16,
    AUTHOR = {Crippa, Gianluca and Lopes Filho, Milton C. and Miot, Evelyne
              and Nussenzveig Lopes, Helena J.},
     TITLE = {Flows of vector fields with point singularities and the
              vortex-wave system},
   JOURNAL = {Discrete Contin. Dyn. Syst.},
  FJOURNAL = {Discrete and Continuous Dynamical Systems. Series A},
    VOLUME = {36},
      YEAR = {2016},
    NUMBER = {5},
     PAGES = {2405--2417},
      ISSN = {1078-0947},
   MRCLASS = {35Q31 (76B47)},
  MRNUMBER = {3485403},
MRREVIEWER = {Franck Sueur},
       DOI = {10.3934/dcds.2016.36.2405},
       URL = {https://doi.org/10.3934/dcds.2016.36.2405},
}

@article {DeLellis08,
    AUTHOR = {De Lellis, Camillo},
     TITLE = {O{DE}s with {S}obolev coefficients: the {E}ulerian and the
              {L}agrangian approach},
   JOURNAL = {Discrete Contin. Dyn. Syst. Ser. S},
  FJOURNAL = {Discrete and Continuous Dynamical Systems. Series S},
    VOLUME = {1},
      YEAR = {2008},
    NUMBER = {3},
     PAGES = {405--426},
      ISSN = {1937-1632},
   MRCLASS = {34A12 (35B30 35F05 35F20 35L65 58D25)},
  MRNUMBER = {2425023},
MRREVIEWER = {Marta Lewicka},
       DOI = {10.3934/dcdss.2008.1.405},
       URL = {https://doi.org/10.3934/dcdss.2008.1.405},
}

@incollection {DeLellis08a,
    AUTHOR = {De Lellis, Camillo},
     TITLE = {Ordinary differential equations with rough coefficients and
              the renormalization theorem of {A}mbrosio [after {A}mbrosio,
              {D}i{P}erna, {L}ions]},
      NOTE = {S\'{e}minaire Bourbaki. Vol. 2006/2007},
   JOURNAL = {Ast\'{e}risque},
  FJOURNAL = {Ast\'{e}risque},
    NUMBER = {317},
      YEAR = {2008},
     PAGES = {Exp. No. 972, viii, 175--203},
      ISSN = {0303-1179},
      ISBN = {978-2-85629-253-2},
   MRCLASS = {35L03 (35L45)},
  MRNUMBER = {2487734},
MRREVIEWER = {Sergio Spagnolo},
}

@incollection {Depauw03,
    AUTHOR = {Depauw, Nicolas},
     TITLE = {Non-unicit\'{e} du transport par un champ de vecteurs presque
              {BV}},
 BOOKTITLE = {Seminaire: \'{E}quations aux {D}\'{e}riv\'{e}es {P}artielles, 2002--2003},
    SERIES = {S\'{e}min. \'{E}qu. D\'{e}riv. Partielles},
     PAGES = {Exp. No. XIX, 9},
 PUBLISHER = {\'{E}cole Polytech., Palaiseau},
      YEAR = {2003},
   MRCLASS = {35F20 (35L60)},
  MRNUMBER = {2030714},
MRREVIEWER = {Thierry Goudon},
}

@article {DiPernaLions89a,
    AUTHOR = {DiPerna, Ronald J. and Lions, Pierre-Louis},
     TITLE = {Ordinary differential equations, transport theory and
              {S}obolev spaces},
   JOURNAL = {Invent. Math.},
  FJOURNAL = {Inventiones Mathematicae},
    VOLUME = {98},
      YEAR = {1989},
    NUMBER = {3},
     PAGES = {511--547},
      ISSN = {0020-9910},
   MRCLASS = {34A10 (34D20 35Q20 58D25 82A70)},
  MRNUMBER = {1022305},
MRREVIEWER = {B. G. Pachpatte},
       DOI = {10.1007/BF01393835},
       URL = {https://doi.org/10.1007/BF01393835},
}

@article {CapuzzoDolcettaPerthame96,
    AUTHOR = {Capuzzo Dolcetta, Italo and Perthame, Beno\^it},
     TITLE = {On some analogy between different approaches to first order
              {PDE}'s with nonsmooth coefficients},
   JOURNAL = {Adv. Math. Sci. Appl.},
  FJOURNAL = {Advances in Mathematical Sciences and Applications},
    VOLUME = {6},
      YEAR = {1996},
    NUMBER = {2},
     PAGES = {689--703},
      ISSN = {1343-4373},
   MRCLASS = {35F10 (35D05 35R05 49L25)},
  MRNUMBER = {1411988},
MRREVIEWER = {Sergey Nikolaevich Alekseenko},
}

@book {EvansGariepy92,
    AUTHOR = {Evans, Lawrence C. and Gariepy, Ronald F.},
     TITLE = {Measure theory and fine properties of functions},
    SERIES = {Studies in Advanced Mathematics},
 PUBLISHER = {CRC Press, Boca Raton, FL},
      YEAR = {1992},
     PAGES = {viii+268},
      ISBN = {0-8493-7157-0},
   MRCLASS = {28-02 (26-02 26Bxx 46E35)},
  MRNUMBER = {1158660},
MRREVIEWER = {R. G. Bartle},
}

@book {BkFalconer03,
    AUTHOR = {Falconer, Kenneth},
     TITLE = {Fractal geometry},
   EDITION = {Second},
      NOTE = {Mathematical foundations and applications},
 PUBLISHER = {John Wiley \& Sons, Inc., Hoboken, NJ},
      YEAR = {2003},
     PAGES = {xxviii+337},
      ISBN = {0-470-84861-8},
   MRCLASS = {28-01 (00A69 11K55 28A75 28A78 28A80 37C45 37F10)},
  MRNUMBER = {2118797},
MRREVIEWER = {Esa J\"{a}rvenp\"{a}\"{a}},
       DOI = {10.1002/0470013850},
       URL = {https://doi.org/10.1002/0470013850},
}

@article {Fefferman,
    AUTHOR = {Fefferman, Charles L. and Pooley, Benjamin C. and Rodrigo,
              Jos\'{e} L.},
     TITLE = {Non-conservation of dimension in divergence-free solutions of
              passive and active scalar systems},
   JOURNAL = {Arch. Ration. Mech. Anal.},
  FJOURNAL = {Archive for Rational Mechanics and Analysis},
    VOLUME = {242},
      YEAR = {2021},
    NUMBER = {3},
     PAGES = {1445--1478},
      ISSN = {0003-9527},
   MRCLASS = {35Q49 (76B03)},
  MRNUMBER = {4334730},
       DOI = {10.1007/s00205-021-01708-6},
       URL = {https://doi.org/10.1007/s00205-021-01708-6},
}

@article {LeoniMorini07,
    AUTHOR = {Leoni, Giovanni and Morini, Massimiliano},
     TITLE = {Necessary and sufficient conditions for the chain rule in
              {$W^{1,1}_{\rm loc}(\mathbb{R}^N;\mathbb{R}^d)$} and {${\rm BV}_{\rm
              loc}(\mathbb{R}^N;\mathbb{R}^d)$}},
   JOURNAL = {J. Eur. Math. Soc. (JEMS)},
  FJOURNAL = {Journal of the European Mathematical Society (JEMS)},
    VOLUME = {9},
      YEAR = {2007},
    NUMBER = {2},
     PAGES = {219--252},
      ISSN = {1435-9855},
   MRCLASS = {46E35 (26B05 26B30 26B40 28A75)},
  MRNUMBER = {2293955},
MRREVIEWER = {Stanislav Hencl},
       DOI = {10.4171/JEMS/78},
       URL = {https://doi.org/10.4171/JEMS/78},
}

@article {Lions98,
    AUTHOR = {Lions, Pierre-Louis},
     TITLE = {Sur les \'{e}quations diff\'{e}rentielles ordinaires et les \'{e}quations
              de transport},
   JOURNAL = {C. R. Acad. Sci. Paris S\'{e}r. I Math.},
  FJOURNAL = {Comptes Rendus de l'Acad\'{e}mie des Sciences. S\'{e}rie I.
              Math\'{e}matique},
    VOLUME = {326},
      YEAR = {1998},
    NUMBER = {7},
     PAGES = {833--838},
      ISSN = {0764-4442},
   MRCLASS = {34A12 (34C35 35Q99)},
  MRNUMBER = {1648524},
MRREVIEWER = {Beno\^{\i}t P. Desjardins},
       DOI = {10.1016/S0764-4442(98)80022-0},
       URL = {https://doi.org/10.1016/S0764-4442(98)80022-0},
}

@incollection {MarPul1,
    AUTHOR = {Marchioro, Carlo and Pulvirenti, Mario},
     TITLE = {On the vortex-wave system},
 BOOKTITLE = {Mechanics, analysis and geometry: 200 years after {L}agrange},
    SERIES = {North-Holland Delta Ser.},
     PAGES = {79--95},
 PUBLISHER = {North-Holland, Amsterdam},
      YEAR = {1991},
   MRCLASS = {35Q30 (76C05)},
  MRNUMBER = {1098512},
MRREVIEWER = {Andro Mikeli\'{c}},
}

@book {MarPul,
    AUTHOR = {Marchioro, Carlo and Pulvirenti, Mario},
     TITLE = {Mathematical theory of incompressible nonviscous fluids},
    SERIES = {Applied Mathematical Sciences},
    VOLUME = {96},
 PUBLISHER = {Springer-Verlag, New York},
      YEAR = {1994},
     PAGES = {xii+283},
      ISBN = {0-387-94044-8},
   MRCLASS = {76-02 (35Q35 76Cxx 76E99 76F99 76M25)},
  MRNUMBER = {1245492},
MRREVIEWER = {J. Thomas Beale},
       DOI = {10.1007/978-1-4612-4284-0},
       URL = {https://doi.org/10.1007/978-1-4612-4284-0},
}

@article {MarcusMizel72,
    AUTHOR = {Marcus, Michael B. and Mizel, Victor J.},
     TITLE = {Absolute continuity on tracks and mappings of {S}obolev
              spaces},
   JOURNAL = {Arch. Rational Mech. Anal.},
  FJOURNAL = {Archive for Rational Mechanics and Analysis},
    VOLUME = {45},
      YEAR = {1972},
     PAGES = {294--320},
      ISSN = {0003-9527},
   MRCLASS = {46E35},
  MRNUMBER = {338765},
MRREVIEWER = {A. Kufner},
       DOI = {10.1007/BF00251378},
       URL = {https://doi.org/10.1007/BF00251378},
}

@article {Nguyen,
    AUTHOR = {Nguyen, Quoc-Hung},
     TITLE = {Quantitative estimates for regular {L}agrangian flows with
              {$BV$} vector fields},
   JOURNAL = {Comm. Pure Appl. Math.},
  FJOURNAL = {Communications on Pure and Applied Mathematics},
    VOLUME = {74},
      YEAR = {2021},
    NUMBER = {6},
     PAGES = {1129--1192},
      ISSN = {0010-3640},
   MRCLASS = {34G20 (37J06 42B20)},
  MRNUMBER = {4242824},
       DOI = {10.1002/cpa.21992},
       URL = {https://doi.org/10.1002/cpa.21992},
}

@book {BkRobinson01,
    AUTHOR = {Robinson, James C.},
     TITLE = {Infinite-dimensional dynamical systems},
    SERIES = {Cambridge Texts in Applied Mathematics},
      NOTE = {An introduction to dissipative parabolic PDEs and the theory
              of global attractors},
 PUBLISHER = {Cambridge University Press, Cambridge},
      YEAR = {2001},
     PAGES = {xviii+461},
      ISBN = {0-521-63204-8},
   MRCLASS = {37-02 (34G10 34L30 35B41 35K57 35Q30 37Lxx)},
  MRNUMBER = {1881888},
MRREVIEWER = {Ricardo M. S. Rosa},
       DOI = {10.1007/978-94-010-0732-0},
       URL = {https://doi.org/10.1007/978-94-010-0732-0},
}

@article {RobinsonSadowski09a,
    AUTHOR = {Robinson, James C. and Sadowski, Witold},
     TITLE = {A criterion for uniqueness of {L}agrangian trajectories for
              weak solutions of the 3{D} {N}avier-{S}tokes equations},
   JOURNAL = {Comm. Math. Phys.},
  FJOURNAL = {Communications in Mathematical Physics},
    VOLUME = {290},
      YEAR = {2009},
    NUMBER = {1},
     PAGES = {15--22},
      ISSN = {0010-3616},
   MRCLASS = {35Q30 (76D05)},
  MRNUMBER = {2520506},
       DOI = {10.1007/s00220-009-0819-z},
       URL = {https://doi.org/10.1007/s00220-009-0819-z},
}

@article {RobinsonSadowski09,
    AUTHOR = {Robinson, James C. and Sadowski, Witold},
     TITLE = {Almost-everywhere uniqueness of {L}agrangian trajectories for
              suitable weak solutions of the three-dimensional
              {N}avier-{S}tokes equations},
   JOURNAL = {Nonlinearity},
  FJOURNAL = {Nonlinearity},
    VOLUME = {22},
      YEAR = {2009},
    NUMBER = {9},
     PAGES = {2093--2099},
      ISSN = {0951-7715},
   MRCLASS = {35Q30 (76D05)},
  MRNUMBER = {2534294},
       DOI = {10.1088/0951-7715/22/9/002},
       URL = {https://doi.org/10.1088/0951-7715/22/9/002},
}

@Article{RobinsonSadowskiSharples12,
  Title                    = {{On the regularity of Lagrangian trajectories corresponding to suitable weak solutions of the Navier-Stokes equations}},
  Author                   = {Robinson, James C. and Sadowski, Witold and Sharples, Nicholas},
  Journal                  = {{Procedia IUTAM}},
  Year                     = {{2012}}
}

@article {RobinsonSharples13JDE,
    AUTHOR = {Robinson, James C. and Sharples, Nicholas},
     TITLE = {Dimension prints and the avoidance of sets for flow solutions
              of non-autonomous ordinary differential equations},
   JOURNAL = {J. Differential Equations},
  FJOURNAL = {Journal of Differential Equations},
    VOLUME = {254},
      YEAR = {2013},
    NUMBER = {10},
     PAGES = {4144--4167},
      ISSN = {0022-0396},
   MRCLASS = {34A26 (28A80)},
  MRNUMBER = {3032300},
MRREVIEWER = {Roland Rabanal},
       DOI = {10.1016/j.jde.2013.02.012},
       URL = {https://doi.org/10.1016/j.jde.2013.02.012},
}

@article {RobinsonSharples13RAEX,
    AUTHOR = {Robinson, James C. and Sharples, Nicholas},
     TITLE = {Strict inequality in the box-counting dimension product
              formulas},
   JOURNAL = {Real Anal. Exchange},
  FJOURNAL = {Real Analysis Exchange},
    VOLUME = {38},
      YEAR = {2012/13},
    NUMBER = {1},
     PAGES = {95--119},
      ISSN = {0147-1937},
   MRCLASS = {28A78 (28A80)},
  MRNUMBER = {3083200},
MRREVIEWER = {Darko \v{Z}ubrini\'{c}},
       URL = {http://projecteuclid.org/euclid.rae/1367265642},
}

@article {Rogers88,
    AUTHOR = {Rogers, C. Ambrose},
     TITLE = {Dimension prints},
   JOURNAL = {Mathematika},
  FJOURNAL = {Mathematika. A Journal of Pure and Applied Mathematics},
    VOLUME = {35},
      YEAR = {1988},
    NUMBER = {1},
     PAGES = {1--27},
      ISSN = {0025-5793},
   MRCLASS = {28A75},
  MRNUMBER = {962731},
MRREVIEWER = {K. J. Falconer},
       DOI = {10.1112/S0025579300006239},
       URL = {https://doi.org/10.1112/S0025579300006239},
}

@unpublished{SerrinUnpublished,
  Title						= {{Unpublished}},
  Author					= {{Serrin, J}}
}

@article {SerrinVarberg69,
    AUTHOR = {Serrin, James and Varberg, Dale E.},
     TITLE = {A general chain rule for derivatives and the change of
              variables formula for the {L}ebesgue integral},
   JOURNAL = {Amer. Math. Monthly},
  FJOURNAL = {American Mathematical Monthly},
    VOLUME = {76},
      YEAR = {1969},
     PAGES = {514--520},
      ISSN = {0002-9890},
   MRCLASS = {26.46},
  MRNUMBER = {247011},
MRREVIEWER = {R. S. Booth},
       DOI = {10.2307/2316959},
       URL = {https://doi.org/10.2307/2316959},
}

@article {Stampacchia65,
    AUTHOR = {Stampacchia, Guido},
     TITLE = {Le probl\`eme de {D}irichlet pour les \'{e}quations elliptiques du
              second ordre \`a coefficients discontinus},
   JOURNAL = {Ann. Inst. Fourier (Grenoble)},
  FJOURNAL = {Universit\'{e} de Grenoble. Annales de l'Institut Fourier},
    VOLUME = {15},
      YEAR = {1965},
    NUMBER = {fasc. 1},
     PAGES = {189--258},
      ISSN = {0373-0956},
   MRCLASS = {35.45},
  MRNUMBER = {192177},
MRREVIEWER = {E. Magenes},
       URL = {http://www.numdam.org/item?id=AIF_1965__15_1_189_0},
}

@article {starovoitov-existence,
    AUTHOR = {Starovo\u{\i}tov, V. N.},
     TITLE = {Solvability of a problem on the motion of concentrated
              vortices in an ideal fluid},
   JOURNAL = {Dinamika Sploshn. Sredy},
  FJOURNAL = {Institut Gidrodinamiki Sibirskogo Otdeleniya Akademii Nauk
              SSSR. Dinamika Sploshno\u{\i} Sredy},
    NUMBER = {85},
      YEAR = {1988},
     PAGES = {118--136, 165},
      ISSN = {0420-0497},
   MRCLASS = {76C05 (35Q10)},
  MRNUMBER = {1003448},
}

@article {starovoitov-unique,
    AUTHOR = {Starovo\u{\i}tov, V. N.},
     TITLE = {Uniqueness of the solution to the problem of the motion of a
              point vortex},
   JOURNAL = {Sibirsk. Mat. Zh.},
  FJOURNAL = {Rossi\u{\i}skaya Akademiya Nauk. Sibirskoe Otdelenie. Institut
              Matematiki im. S. L. Soboleva. Sibirski\u{\i} Matematicheski\u{\i}
              Zhurnal},
    VOLUME = {35},
      YEAR = {1994},
    NUMBER = {3},
     PAGES = {696--701, v},
      ISSN = {0037-4474},
   MRCLASS = {76C05 (35Q30)},
  MRNUMBER = {1292230},
       DOI = {10.1007/BF02104828},
       URL = {https://doi.org/10.1007/BF02104828},
}

\end{document}